\documentclass[a4paper]{amsart}

\usepackage[a4paper,hmargin=3cm,vmargin=4cm]{geometry}
\usepackage{amsfonts,amssymb,amscd,amstext}
\usepackage{graphicx}
\usepackage[dvips]{epsfig}

\usepackage[colorlinks=true,linkcolor=blue,citecolor=red]{hyperref}

\renewcommand{\leq}{\leqslant}
\renewcommand{\geq}{\geqslant}
\newcommand{\ptl}{\partial}
\newcommand{\rr}{{\mathbb{R}}}
\newcommand{\nn}{{\mathbb{N}}}
\newcommand{\la}{\lambda}
\newcommand{\hh}{{\mathbb{H}}}
\newcommand{\sph}{{\mathbb{S}}}
\newcommand{\ele}{\mathcal{L}}
\newcommand{\cil}{\mathcal{C}}
\newcommand{\sub}{\subset}
\newcommand{\subeq}{\subseteq}
\newcommand{\escpr}[1]{\big<#1\big>}

\newcommand{\sg}{\sigma}
\newcommand{\Om}{\Omega}

\newcommand{\var}{\varphi}

\newcommand{\ric}{\text{Ric}}
\newcommand{\ovh}{\overline{H}_P}
\newcommand{\C}{\operatorname{Cap}}

\DeclareMathOperator{\divv}{div}

\newtheorem{theorem}{Theorem}[section]
\newtheorem{proposition}[theorem]{Proposition}
\newtheorem{lemma}[theorem]{Lemma}
\newtheorem{corollary}[theorem]{Corollary}

\theoremstyle{definition}
\newtheorem{remark}[theorem]{Remark}
\newtheorem{remarks}[theorem]{Remarks}
\newtheorem{example}[theorem]{Example}
\newtheorem{examples}[theorem]{Examples}

\numberwithin{equation}{section}

\begin{document}

\title[Parabolicity criteria and characterization of submanifolds]
{Parabolicity criteria and characterization results \\ for submanifolds of bounded mean curvature \\ in model manifolds with weights}

\author[A. Hurtado]{A. Hurtado$^{\natural}$}
\address{Departamento de Geometr\'{\i}a y Topolog\'{\i}a and Excellence Research Unit ``Modeling Nature'' (MNat), Universidad de Granada, E-18071,
Spain.}
 \email{ahurtado@ugr.es}

\author[V. Palmer]{V. Palmer*}
\address{Departament de Matem\`{a}tiques, Universitat Jaume I, Castell\'o,
Spain.} 
\email{palmer@mat.uji.es}

\author[C. Rosales]{C. Rosales$^{\#}$}
\address{Departamento de Geometr\'{\i}a y Topolog\'{\i}a and Excellence Research Unit ``Modeling Nature'' (MNat) Universidad de Granada, E-18071,
Spain.} 
\email{crosales@ugr.es}

\date{\today}

\thanks{* Supported by MINECO grant No.~MTM2017-84851-C2-2-P and UJI grant UJI-B2016-07. \\
\indent $^{\natural}$ $^{\#}$ Supported by MINECO grant No.~
MTM2017-84851-C2-1-P and Junta de Andaluc\'ia grant No.~FQM325} 



\begin{abstract} 
Let $P$ be a submanifold properly immersed in a rotationally symmetric manifold having a pole and endowed with a weight $e^h$. The aim of this paper is twofold. First, by assuming certain control on the $h$-mean curvature of $P$, we establish comparisons for the $h$-capacity of extrinsic balls in $P$, from which we deduce criteria ensuring the $h$-parabolicity or $h$-hyperbolicity of $P$. Second, we employ functions with geometric meaning to describe submanifolds of bounded $h$-mean curvature which are confined into some regions of the ambient manifold. As a consequence, we derive half-space and Bernstein-type theorems generalizing previous ones. Our results apply for some relevant $h$-minimal submanifolds appearing in the singularity theory of the mean curvature flow.
\end{abstract}

\maketitle

\thispagestyle{empty}

\section{Introduction}
\label{sec:intro}

\emph{Weighted manifolds} (also known as \emph{manifolds with density}) are Riemannian manifolds where a smooth positive function weights the Hausdorff measures associated to the Riemannian distance. These provide a generalization of Riemannian geometry appearing in different contexts and recently studied by many authors, see Morgan's book \cite[Ch.~18]{gmt} for a nice introduction. 

In this paper we will consider weights in rotationally symmetric manifolds with a pole (that we call \emph{model spaces}, see Section~\ref{subsec:models} for a precise definition). We will follow two objectives. The first one is to establish some criteria ensuring parabolicity or hyperbolicity of submanifolds in a weighted setting. These will be obtained from capacity comparisons with respect to \emph{weighted model spaces}, which are model spaces together with a \emph{radial weight}. The second objective is to prove geometric restrictions and characterization theorems for submanifolds with controlled weighted mean curvature. These will be deduced in a unified way by applying the defining property of parabolic submanifolds with suitable functions having a geometric meaning. In order to describe our results in more detail we need to introduce some background and motivations.  

Following the classical definition in potential theory, the \emph{$h$-parabolicity} condition in a Riemannian manifold $M$ with a weight $e^h$ is the Liouville-type property that any function $u$ bounded from above and \emph{$h$-subharmonic}, i.e., $\Delta^hu\geq 0$ must be constant. Here $\Delta^h u$ is the \emph{weighted Laplacian operator}, which is defined in Grigor'yan~\cite[Sect.~2.1]{grigoryan2} as the sum of the Laplace-Beltrami operator $\Delta u$ in $M$ and the first order term $\escpr{\nabla h,\nabla u}$. If the parabolicity condition fails then $M$ is said to be \emph{$h$-hyperbolic}. Clearly, for a constant weight these notions agree with the standard ones for Riemannian manifolds. Indeed, as happens in the unweighted setting, the $h$-parabolicity of $M$ is characterized by the fact that $\text{Cap}^h(K)=0$ for any / some compact set $K\subeq M$ with non-empty interior, see Grigor'yan and Saloff-Coste~\cite[Sect.~1.7]{grigoryan3} and the references therein. Here $\text{Cap}^h(K)$ denotes the \emph{$h$-capacity} of $K$ defined as $\lim_{k\to\infty}\text{Cap}^h(K,\Omega_k)$, where $\{\Om_k\}_{k\in\nn}$ is any exhaustion of $M$ by smooth precompact open sets, and $\text{Cap}^h(K,\Om_k)$ is the $h$-capacity defined in \eqref{eq:capacity} of the capacitor $(K,\Om_k)$. The value of $\text{Cap}^h(K,\Om_k)$ can be calculated by means of equality \eqref{eq:capint} involving the \emph{$h$-capacity potential}, which is the solution to the weighted Laplace equation with Dirichlet boundary condition appearing in \eqref{eq:laplace}. This characterization is extremely useful and allows to deduce parabolicity criteria for weighted manifolds relying on growth properties of the weighted measures, see for instance Grigor'yan \cite[Sect.~9.1]{grigoryan2}, and Grigor'yan and Masamune~\cite[Sect.~1]{gri-masa}. In particular, $M$ is $h$-parabolic provided it has finite weighted volume. In the case of a weighted model space it is possible to compute the $h$-capacity for open metric balls centered at the pole, see Proposition~\ref{prop:epm}. As a consequence, we recover in Corollary~\ref{cor:ahlfors} an Ahlfors-type result, already showed by Grigor'yan~\cite[Ex.~9.5]{grigoryan2}, which describes the parabolicity of a weighted model space by means of an integrability condition for the weighted area function of the metric spheres centered at the pole. As an application of this principle we illustrate that a non-compact manifold $M$ may have weights $e^{h_1}$ and $e^{h_2}$ such that $M$ is, at the same time, $h_1$-parabolic and $h_2$-hyperbolic.

In contrast to the aforementioned criteria, which are of \emph{intrinsic} nature, in this paper we study the parabolicity property from the \emph{extrinsic} point of view. More precisely, given a submanifold with empty boundary $P$ properly immersed in a Riemannian manifold $M$ with a weight $e^h$, we seek sufficient restrictions on $h$ and the weighted extrinsic geometry of $P$ to guarantee that $P$ is $h$-parabolic or $h$-hyperbolic when we see it as a weighted manifold with respect to the Riemannian metric and weight inherited from $M$. In the Riemannian setting there are several results in this line, see the expository article \cite{Pa2} of the second author for a very complete account. These results are derived from a geometric analysis of the distance function which allows to estimate the Laplacian of radial functions by assuming suitable bounds on the mean curvature vector of $P$ and the intrinsic curvatures of $M$. One of our motivations in this work is to generalize this technique in order to prove similar statements for weighted manifolds, then showing the fundamental role of the weighted mean curvature vector in determining the $h$-parabolicity or $h$-hyperbolicity of $P$.

The \emph{weighted mean curvature vector} $\ovh^h$ of a submanifold $P$ is 
the normal vector to $P$ defined in \eqref{eq:mcvector} in terms of the Riemannian mean curvature vector $\ovh$ and the normal projection to $P$ of $\nabla h$. This was first introduced for two-sided hypersurfaces by Gromov~\cite[Sect.~9.4.E]{gromov-GAFA}, see also Bayle~\cite[Sect.~3.4.2]{bayle-thesis}, by means of variational arguments. As in Riemannian geometry, the weighted mean curvature controls the weighted extrinsic geometry of $P$, and it is interesting for several reasons. On the one hand, see for instance \cite[Sect.~3]{rcbm} and \cite[Sect.~3]{castro-rosales}, hypersurfaces with vanishing (resp. constant) weighted mean curvature are the critical points for the Plateau problem (resp. isoperimetric problem), where we try to minimize the weighted area under a constraint on the boundary (resp. weighted volume). On the other hand, there is an important connection between \emph{weighted minimal submanifolds} (those with $\ovh^h=0$) and the singularity theory of the mean curvature flow. It was observed by Colding and Minicozzi~\cite[Lem.~2.2]{colding-minicozzi}, \cite[Sect.~1.1]{cm2} that the self-similar solutions to the mean curvature flow in $\rr^m$ called \emph{self-shrinkers} satisfy equation $(m-1)\,H_P(p)=\escpr{p,N(p)}$ involving the Euclidean mean curvature $H_P$ and the normal $N$ of a hypersurface $P$. Thus, it follows from \eqref{eq:fmc} that the self-shrinker hypersurfaces coincide with the minimal hypersurfaces in $\rr^m$ with respect to the Gaussian weight $e^{-|p|^2/2}$. In a similar way, the self-similar solutions called \emph{self-expanders}, and those hypersurfaces for which the evolution under the flow is given by translations (the so-called \emph{translating solitons}), are minimal submanifolds in $\rr^m$ with a suitable weight, see Example~\ref{ex:singularity} for more details. 

Coming back to the relation between the weighted mean curvature vector $\ovh^h$ and the $h$-parabolicity of $P$, we must mention some previous results in this direction. In \cite[Thm.~4.1]{cheng-zhou-volume}, Cheng and Zhou proved that an $n$-dimensional self-shrinker properly immersed in $\rr^m$ has finite weighted $n$-dimensional volume and so, it is weighted parabolic. This fact was later extended by Cheng, Mejia and Zhou~\cite[Cor.~1]{cmz} to weighted minimal submanifolds properly immersed in certain shrinking gradient Ricci solitons (complete weighted manifolds such that $\text{Ric}-\nabla^2h=c\,g$, where $\nabla^2$ denotes the Hessian, $\text{Ric}$ is the Ricci tensor, $g$ is the Riemannian metric and $c$ is a positive constant). As recently shown by Alencar and Rocha~\cite[Thm.~1]{alencar-rocha}, the result still holds for a properly immersed submanifold $P$ where the function $\escpr{\ovh^h,\nabla h}$ is bounded from above. 

In Section~\ref{sec:criteria} we provide new parabolicity and hyperbolicity criteria for submanifolds in a weighted context. The main statements are contained in Theorems~\ref{parabolicity-sub-model} and \ref{hyperbolicity-sub-model}, which might be seen as weighted counterparts to analogous results for Riemannian manifolds proved by Esteve and the second author~\cite[Thm.~3.4]{esteve-palmer}, and by Markvorsen and the second author,  see \cite[Thm.~2.1]{mp-GAFA} and \cite[Thm.~A]{mp-transience}, respectively. In order to motivate the hypotheses of both theorems we will briefly explain their proof. Given a submanifold $P$ properly immersed in a model space with weight $e^h$, we want to estimate the $h$-capacity over the exhaustion of $P$ associated to extrinsic balls centered at the pole. For this, we compare the $h$-capacity potentials for extrinsic capacitors in $P$ associated to concentric balls with the radial functions obtained by transplanting to $P$, via the distance function $r(p)$ with respect to the pole, the capacity potentials computed in Proposition~\ref{prop:epm} for intrinsic capacitors associated to concentric balls in a suitable comparison weighted model. To choose such a model we make use of Lemma~\ref{equality-laplace-h}, which illustrates how to control the weighted Laplacian of radial functions restricted to $P$ by means of a radial estimate for the function $\escpr{\nabla h,\nabla r}+\escpr{\ovh^h,\nabla r}$, which involves the radial components of the vector fields $\nabla h$ and $\ovh^h$, and a balance condition between this radial estimate and the Riemannian mean curvature of the metric spheres centered at the pole. Once the comparison model is determined we invoke the Ahlfors-type criterion in Corollary~\ref{cor:ahlfors} to infer the $h$-parabolicity or $h$-hyperbolicity of $P$ from the corresponding integrability hypothesis for the weighted area of the spheres centered at the pole.

Our parabolicity and hyperbolicity criteria can be applied to several interesting situations. They are valid for submanifolds of any codimension properly immersed in model spaces (like Euclidean space, hyperbolic space or convex paraboloids of revolution), and having bounded mean curvature vector with respect to some relevant weights. In Corollaries~\ref{cor:parabolicity} and \ref{cor:radial2} we provide some consequences for (eventually perturbed) radial weights. In $\rr^m$ these weights have received an increasing attention in the last years, specially in relation to isoperimetric problems and rigidity properties for the self-similar solutions of the mean curvature flow, see for instance \cite{lcdc}, \cite{cao-li}, \cite{rosales-gauss}, \cite{bcm3}, \cite{brendle}, \cite{cz-expanders} and \cite{chambers}. From Corollary~\ref{cor:parabolicity} we deduce the $h$-parabolicity (resp. $h$-hyperbolicity) of any non-compact submanifold properly immersed in Euclidean space $\rr^m$ or hyperbolic space $\mathbb{H}^m$, and having bounded mean curvature with respect to a radial weight $e^{f(r)}$ such that $f'(t)\to-\infty$ (resp. $f(t)\to\infty$) when $t\to\infty$. In the particular case of $\rr^m$ with Gaussian weight $e^{-|p|^2/2}$, this extends the previously mentioned parabolicity result of Cheng and Zhou for self-shrinkers \cite[Thm.~4.1]{cheng-zhou-volume}. Indeed, it also follows that hypersurfaces of constant weighted mean curvature $\la$ for the Gaussian weight (usually called \emph{$\la$-hypersurfaces}) are weighted parabolic. Moreover, in $\rr^m$ with (anti)Gaussian weight $e^{|p|^2/2}$ our result, together with the non-existence of compact self-expanders, see Cao and Li~\cite[Prop.~5.3]{cao-li}, implies that any properly immersed self-expander is weighted hyperbolic. It is interesting to observe that our criteria entail the existence of several weights in $\rr^m$ and $\mathbb{H}^m$ (as the radial ones above) for which \emph{all the properly immersed submanifolds with bounded $h$-mean curvature are $h$-parabolic, independently of their dimension}. This is in clear contrast to the unweighted setting, where the parabolicity condition for non-compact submanifolds is much more restrictive. For example, it is known by a result of Markvorsen and the second author~\cite[Thm.~2.1]{mp-GAFA} that all the minimal submanifolds in $\rr^m$ (resp. $\mathbb{H}^m$) of dimension $n\geq 3$ (resp. $n\geq 2$) are hyperbolic. 

Once we have shown abundance of parabolic submanifolds, our next aim is to infer information about them by employing the Liouville-type property with suitable geometric functions. In Section~\ref{sec:char} we follow a unified approach to prove rigidity properties for submanifolds of arbitrary codimension in model spaces with certain weights. We will consider three different situations: submanifolds confined into some regions of a model space, entire horizontal graphs in Euclidean space, and two-sided hypersurfaces satisfying a stability condition in a model space. 

In Section~\ref{subsec:ball} we study submanifolds inside or outside a metric ball $B_{t_0}$ centered at the pole of a model space. Previous related results for complete self-shrinkers immersed in some Euclidean balls were derived by Vieira and Zhou~\cite[Thm.~1]{vieira-zhou}, Pigola and Rimoldi~\cite[Thm.~1]{pigola-rimoldi} and Gimeno and the second author~\cite[Thm.~6.1]{gimeno-palmer-preprint} by assuming other hypotheses. It is easy to see that a metric sphere $S_t$ about the pole is a hypersurface of constant mean curvature with respect to any radial weight. In Theorems~\ref{th:ball1} and \ref{th:ball2} we provide sufficient conditions on an $n$-dimensional submanifold $P$ properly immersed inside or outside a metric ball $B_{t_0}$ of a weighted model space to conclude that $P$ is contained in a metric sphere $S_t$. Our conditions rely on suitable bounds for the mean curvature vector $\ovh^h$, and for the mean curvature of metric spheres in the corresponding $n$-dimensional weighted model space. These guarantee not only the $h$-parabolicity of $P$, which is a consequence of our parabolicity criteria, but also that a certain radial function $v$ on $P$ (which in $\rr^m$ coincides with the squared distance $r^2/2$), is $h$-subharmonic or $h$-superharmonic. From both facts we get that $v$ must be constant, which proves the statements. By using the same arguments we deduce in Corollary~\ref{cor:ball5} that, for certain radial weights having a singularity at the pole, the only compact $h$-minimal hypersurfaces avoiding the pole are the metric spheres $S_t$. In the particular case of the homogeneous weight $r^{1-m}$ in $\rr^m$ this improves a result for hypersurfaces of Ca\~nete and the third author~\cite[Thm.~6.4]{homostable}.

In Section~\ref{subsec:cylinder} we consider Euclidean space $\rr^m=\rr^k\times\rr^{m-k}$ together with perturbations of the Gaussian weight for which the cylindrical hypersurfaces $\mathcal{C}_t:=S_t\times\rr^{m-k}$ have constant weighted mean curvature. By following the approach in Section~\ref{subsec:ball} we show in Theorem~\ref{th:cylinder} that some submanifolds of bounded weighted mean curvature vector and properly immersed inside or outside a solid cylinder $B_{t_0}\times\rr^{m-k}$ must be contained in a cylinder $\mathcal{C}_t$. The geometric functions employed in this context are $d$ and $d^2$, where $d$ is the horizontal norm $d(x,y):=|x|$ in $\rr^k\times\rr^{m-k}$. In the particular example of the Gaussian weight our theorem provides a different proof of some results established by Cavalcante and Espinar~\cite{espinar-halfspace} for $\lambda$-hypersurfaces, see also Pigola and Rimoldi~\cite[Thm.~2]{pigola-rimoldi}, and Impera, Pigola and Rimoldi~\cite[Thm.~A]{ipr} for the case of self-shrinker hypersurfaces.

The well-known half-space theorem of Hoffman and Meeks \cite{hoffman-meeks} states that a minimal surface properly immersed in a closed half-space of $\rr^3$ must be a plane. In Section~\ref{subsec:half-space} we analyze the height function with respect to a Euclidean hyperplane in order to prove an analogous result in $\rr^m$ with suitable weights. In the case of a radial weight the linear hyperplanes are weighted minimal and so, it is natural to ask if any weighted minimal hypersurface properly immersed in a linear closed half-space must coincide with the boundary of such half-space. This question has a positive answer in $\rr^m$ with Gaussian weight, as was shown by Pigola and Rimoldi~\cite[Thm.~3]{pigola-rimoldi}, see also Cavalcante and Espinar~\cite[Thm.~1.1]{espinar-halfspace}. In Theorem~\ref{th:half-space1} we extend this fact to submanifolds which are minimal for radial weights $e^{f(r)}$ such that $f$ is a decreasing function and $f'(t)\to-\infty$ when $t\to\infty$. Moreover, for some perturbations of the Gaussian weight in $\rr^m$, we are able to deduce in Theorem~\ref{th:half-space2} a result that Cavalcante and Espinar~\cite[Thm.~1.4]{espinar-halfspace} proved for $\lambda$-hypersurfaces properly immersed in a closed half-space of $\rr^m$ with boundary of weighted mean curvature $\la$.

The half-space theorem has been also investigated in Riemannian cylinders $M\times\rr$ with some product weights, see Cavalcante, de Lima and Santos~\cite{cavalcante-santos}, and de Lima and Santos~\cite{delima-santos}. In Theorem~\ref{th:half-space3} we establish some results in this line for weighted parabolic submanifolds in Euclidean space $\rr^m=\rr^{m-1}\times\rr$ endowed with a product weight $e^h$ with $h(x,t):=\mu(t)$ for some function $\mu\in C^1(\rr)$. As a consequence, we get in Corollary~\ref{cor:3.2} several hyperbolicity criteria for $h$-minimal submanifolds within a closed half-space. These criteria apply in particular for the weight $e^t$, for which the associated minimal hypersurfaces are the translating solitons of the mean curvature flow.

The Bernstein problem in $\rr^m=\rr^{m-1}\times\rr$ seeks smooth entire graphs $t=\var(x)$ that solve the minimal surface equation. It is well-known that the unique solutions to this problem are Euclidean hyperplanes if and only if $m\leq 7$, see Giusti~\cite[Ch.~17]{giusti}. In $\rr^m$ with Gaussian weight an entire minimal graph must be a linear hyperplane. For self-shrinkers with polynomial growth this comes from the work of Ecker and Huisken~\cite[App.]{ecker-huisken}, whereas Espinar~\cite[Thm.~4.2]{espinar} studied the case of weighted parabolic self-shrinkers. The additional hypotheses were removed by Wang~\cite{wang}, who proved the general case. It is interesting to mention that, in contrast to the unweighted setting, the solution to the Bernstein problem for the Gaussian weight does not depend on the dimension. Another related result was given by Doan and Tran~\cite{doan-bernstein}, who showed that in $\rr^m=\rr^{m-1}\times\rr$ with mixed Gaussian-Euclidean weight, the only entire horizontal minimal graphs are the horizontal hyperplanes $\rr^{m-1}\times\{t\}$. This was later extended by de Lima, Oliveira and Santos~\cite[Cor.~1]{limaetal} for weighted parabolic entire graphs of constant weighted mean curvature. On the other hand, for translating solitons in $\rr^m$ that are entire graphs, a theorem of Bao and Shi \cite{baoshi} implies that they must be hyperplanes provided the Gauss map image lies in a compact set of an open spherical hemisphere. Coming back to the Gaussian weight, where all the affine hyperplanes have constant weighted mean curvature, it is also natural to ask if these hyperplanes are the unique smooth entire graphs of constant $h$-mean curvature. This question was positively answered by Cavalcante, de Lima and Santos~\cite[Cor.~4]{cavalcante-santos} by assuming an $L^1$ integrability hypothesis on the gradient of the graph. The general case was settled by Cheng and Wei~\cite[Thm.~1.3]{cheng-wei} as a consequence of their study of the Gauss map for properly immersed $\lambda$-hypersurfaces. Recently Doan~\cite{doan-bernstein2} has given another proof based on the isoperimetric property of hyperplanes in Gauss space. 

In Section~\ref{subsec:bernstein} we establish a new Bernstein-type theorem in $\rr^m=\rr^{m-1}\times\rr$ for some product weights. More precisely, in Theorem~\ref{th:bernstein} we provide sufficient conditions ensuring that a smooth entire horizontal graph of constant weighted mean curvature must be a hyperplane. For this, we compute the weighted Laplacian of the angle function $\theta$ between the vertical direction in $\rr^m$ and the unit normal to the graph. Since our conditions guarantee that the graph is $h$-parabolic and $\theta$ is $h$-subharmonic, we can conclude that $\theta$ is constant and the proof easily follows. Theorem~\ref{th:bernstein} applies for some perturbations of the Gaussian weight, thus extending previously mentioned results by means of a different technique.  

We finish this work with a direct application of our parabolicity criteria to the classification of stable hypersurfaces in a weighted sense. A two-sided hypersurface $P$ in a model manifold with weight $e^h$ is \emph{strongly $h$-stable} if it has constant weighted mean curvature $H^h_P$, and it is a second order minimum of the functional $A_h-H^h_P\,V_h$ under compactly supported variations (here $A_h$ and $V_h$ stand for the weighted area and volume functionals). The study of these hypersurfaces has been focus of attention in the last years, with special emphasis in the minimal case, see for instance Fan~\cite{fan}, Ho~\cite{ho}, Colding and Minicozzi~\cite{colding-minicozzi,cm2}, Liu~\cite{liu}, Cheng, Mejia and Zhou~\cite{cmz}, Impera and Rimoldi~\cite{impera}, and Espinar~\cite{espinar}. The $h$-parabolicity condition for a two-sided hypersurface $P$ entails the existence of a sequence of smooth functions with compact support on $P$ approximating the constant function $1$, see Theorem~\ref{th:parchar}. By using these functions together with the stability property, it is straightforward to deduce that an $h$-parabolic and strongly $h$-stable hypersurface must be totally geodesic if the ambient manifold has non-negative Bakry-\'Emery-Ricci curvature. This rigidity principle was previously obtained by Espinar~\cite[Thm.~3.1]{espinar}, who employed gradient Schr\"odinger operators. In Section~\ref{subsec:stable} we make use of this principle and the parabolicity criteria in Section~\ref{sec:criteria} to provide characterization and non-existence results for strongly $h$-stable hypersurfaces, see Theorem~\ref{cor:stable1} and Corollary~\ref{cor:eustable}.

The paper is organized into four sections. The second one mainly contains background material about weighted manifolds, where we gather some facts about potential theory and weighted extrinsic geometry of submanifolds. We also recover the previously mentioned Ahlfors-type description of the parabolicity of a weighted model space. In Section~\ref{sec:criteria} we prove our criteria for the $h$-parabolicity or $h$-hyperbolicity of submanifolds in model spaces with weights. Finally,  Section~\ref{sec:char} is devoted to our rigidity properties and characterization results for submanifolds properly immersed in model spaces and having bounded mean curvature vector with respect to suitable weights.

\section{Preliminaires}
\label{sec:prelimi}

In this section we introduce the notation and gather some basic results that will be used throughout this work.

\subsection{Some potential theory in weighted manifolds}
\label{subsec:potential}
\noindent

Let $M^m$ be a smooth, connected, $m$-dimensional Riemannian manifold with empty boundary. For a function $u\in C^1(M)$ we denote by $\nabla u$ its Riemannian gradient. If $u\in C^2(M)$ then the Hessian at a point $p\in M$ is the bilinear map $(\text{Hess}\,u)_p(X,Y):=\escpr{D_X\nabla u,Y}$, where $\escpr{\cdot\,,\cdot}$ stands for the Riemannian metric, $D$ is the Levi-Civita connection, and $X,Y$ are vectors in the tangent space $T_pM$. The Laplacian of $u$ is the function $\Delta u:=\divv(\nabla u)=\text{tr}(\text{Hess}\,u)$, where $\divv$ is the Riemannian divergence of $C^1$ vector fields on $M$.

A \emph{weighted manifold} is a Riemannian manifold $M^m$ together with a $C^1$ function $e^h$, which is used to weight the Hausdorff measures associated to the Riemannian metric. In particular, for any Borel set $E\subeq M$, and any $C^1$ hypersurface $P\sub M$, the \emph{weighted volume} of $E$ and the \emph{weighted area} of $P$ are given by
\begin{equation}
\label{eq:volarea}
V_h(E):=\int_E dv_h=\int_E e^h\,dv,\quad A_h(P):=\int_P da_h=\int_P e^h\,da,
\end{equation}
where $dv$ and $da$ denote the Riemannian elements of volume and area, respectively.

In weighted manifolds there are generalizations not only of volume and area, but also of some differential operators of Riemannian manifolds. Following Grigor'yan \cite[Sect.~2.1]{grigoryan2}, we define the \emph{weighted Laplacian} or \emph{$h$-Laplacian} of a function $u\in C^2(M)$ as
\begin{equation}
\label{eq:flaplacian}
\Delta^hu:=\Delta u+\escpr{\nabla h,\nabla u}.
\end{equation}
This is a second order linear operator, which is self-adjoint with respect to $dv_h$ since
\[
\int_Mu\,\Delta^hw\,dv_h=\int_Mw\,\Delta^hu\,dv_h,
\]
for any two functions $u,w\in C^2_0(M)$.

Given a domain (connected open set) $\Om$ in $M$, a function $u\in C^2(\Om)$ is \emph{$h$-harmonic} (resp. \emph{$h$-subharmonic}) if $\Delta^h u=0$ (resp. $\Delta^h u\geq 0$) on $\Om$. As in the unweighted setting there is a strong maximum principle and a Hopf boundary point lemma for $h$-subharmonic functions. We gather both results in the next statement.

\begin{theorem}
\label{th:mp}
Let $\Om$ be a smooth domain of a Riemannian manifold $M$ with a weight $e^h$. Consider an $h$-subharmonic function $u\in C^2(\Om)\cap C(\overline{\Om})$. Then, we have:
\begin{itemize}
\item[(i)] if $u$ achieves its maximum in $\Om$ then $u$ is constant,
\item[(ii)] if there is $p_0\in\ptl\Om$ such that $u(p)<u(p_0)$ for any $p\in\Om$ then $\frac{\ptl u}{\ptl\nu}(p_0)>0$, where $\nu$ denotes the outer unit normal along $\ptl\Om$.
\end{itemize}
\end{theorem}

\begin{proof}
The proof of (i) can be found in \cite[Cor.~8.15]{gri-book}. The proof of (ii) can be derived from (i) as in the unweighted Euclidean case \cite[Lem.~3.4]{gilbarg-trudinger} by using a radial barrier comparison function.
\end{proof}

From the maximum principle it is clear that any $h$-subharmonic function on a compact manifold $M$ must be constant. In general, a weighted manifold $M$ is \emph{weighted parabolic} or \emph{$h$-parabolic} if any $h$-subharmonic function which is bounded from above must be constant. Otherwise we say that $M$ is \emph{weighted hyperbolic} or \emph{$h$-hyperbolic}. 

Next, we will recall how the $h$-parabolicity of manifolds can be characterized by means of weighted capacities. For more details about the definitions and the results below we refer the reader to \cite[Sect.~4.3]{grigoryan}, \cite[Sect.~2.3]{grigoryan-escape} and \cite[Sect.~2]{gri-masa}. 

A \emph{capacitor} is a pair $(K,\Om)$ where $\Om\subeq M$ is an open set and $K\sub\Om$ is compact. The \emph{$h$-capacity} of $(K,\Om)$ is the non-negative number given by
\begin{equation}
\label{eq:capacity}
\text{Cap}^h(K,\Om):=\inf\left\{\int_\Om|\nabla\phi|^2\,dv_h\,;\,\phi\in C^\infty_0(\overline{\Om}) \text{ with }0\leq\phi\leq 1\text{ and }\phi=1\text{ on }K\right\},
\end{equation}
where $|X|$ is the norm of a vector field $X$ on $M$.  For a precompact open set $D$ with $\overline{D}\sub\Om$ we denote $\text{Cap}^h(D,\Om):=\text{Cap}^h(\overline{D},\Om)$. We simply write $\text{Cap}^h(K):=\text{Cap}^h(K,M)$ and $\text{Cap}^h(D):=\text{Cap}^h(D,M)$. Clearly $\text{Cap}^h(K,\Om)$ is non-decreasing with respect to $K$ and non-increasing with respect to $\Om$. 
Indeed, we have
\begin{equation}
\label{eq:caplim}
\text{Cap}^h(K)=\lim_{k\to\infty}\text{Cap}^h(K,\Om_k),
\end{equation}
where $\{\Om_k\}_{k\in\nn}$ is any exhaustion of $M$ by precompact open sets with smooth boundaries.

Given a capacitor $(K,\Om)$ where $\Om$ is a smooth precompact open set and $K$ has smooth boundary, it is known that the infimum in $\text{Cap}^h(K,\Om)$ is attained by the \emph{$h$-capacity potential} of $(K,\Om)$, which is the unique solution to the following Dirichlet problem for the weighted Laplace equation
\begin{equation}
\label{eq:laplace}
\begin{cases}
\Delta^h u=0\,\,\,&\text{in\, $\Om \setminus K$},\\
\phantom{\Delta^h }u=1\,\,\,&\text{in\, $\partial K$}, \\
\phantom{\Delta^h }u=0\,\,\,&\text{in\, $\partial \Om$}.
\end{cases}
\end{equation}
As a consequence, we get
\begin{equation}
\label{eq:capint}
\text{Cap}^h(K,\Om)=\int_{\Om\setminus K}|\nabla u|^2\,dv_h=\int_{\ptl K}\frac{\ptl u}{\ptl\nu}\,da_h,
\end{equation}
where $\nu$ is the outer unit normal along $\ptl(\Om\setminus K)$, i.e., the unit normal along $\ptl K$ pointing into $K$. The second equality in \eqref{eq:capint} comes from \eqref{eq:laplace} by using the integration by parts formula (see \cite[Sect.~2]{homostable})
\[
\int_U|\nabla w|^2\,dv_h=-\int_U w\,\Delta^h w\,dv_h+\int_{\ptl U} w\,\frac{\ptl w}{\ptl\nu}\,da_h,
\]
where $U$ is a smooth open set, $\nu$ is the outer unit normal along $\ptl U$ and $w\in C^2_0(\overline{U})$. 

Now, we can state the aforementioned characterization of the $h$-parabolicity by using weighted capacities and suitable approximations of the constant function $1$ on $M$.

\begin{theorem}
\label{th:parchar}
For a Riemannian manifold $M$ with a weight $e^h$, these conditions are equivalent:
\begin{itemize}
\item[(i)] $M$ is $h$-parabolic,
\item[(ii)] $\emph{Cap}^h(K)=0$ for some compact set $K\sub M$ with non-empty interior,
\item[(iii)] $\emph{Cap}^h(K)=0$ for any compact set $K\sub M$,
\item[(iv)] there is a sequence $\{\var_k\}_{k\in\nn}\sub C^\infty_0(M)$ with: 
\begin{enumerate}
\item[(a)] $0\leq\var_k\leq 1$ for any $k\in\nn$, 
\item[(b)] for any compact set $K\sub M$, there is $k_0\in\nn$ such that $\var_k=1$ on $K$ for any $k\geq k_0$, 
\item[(c)] $\lim_{k\to\infty}\int_M|\nabla\var_k|^2\,dv_h=0$.
\end{enumerate}
\end{itemize}
\end{theorem}

\begin{proof}
The equivalence between (i), (ii) and (iii) is found in \cite[Sect.~1.7]{grigoryan3}. That (iii) is equivalent to (iv) comes from \eqref{eq:capacity} by taking an exhaustion $\{\Om_k\}_{k\in\nn}$ of $M$ by precompact open subsets. 
\end{proof}

The previous theorem is used to show parabolicity of non-compact manifolds in some situations. Suppose that $M$ is complete and denote by $B_t(p)$ the open metric ball of radius $t>0$ centered at $p\in M$. As in the Riemannian case, the $h$-parabolicity of $M$ is related to the integrability properties of the function $v_h(t):=V_h(B_t(p))$, see \cite[Sect.~9.1]{grigoryan2} and \cite[p.~608]{gri-masa}. For instance, if $\int_{t_0}^\infty t\,v_h(t)^{-1}dt=\infty$ for some $t_0>0$, then $M$ is $h$-parabolic. This clearly holds when $M$ has finite weighted volume. An important example is the Gaussian weight $e^{-|p|^2/2}$ in $\rr^m$.

\subsection{Submanifolds in weighted manifolds}
\label{subsec:submani}
\noindent

Given a Riemannian manifold $M^m$ we denote by $P^n$ an $n$-dimensional  ($n<m$) smooth submanifold with $\ptl P=\emptyset$ immersed in $M$. We consider in $P$ the induced Riemannian metric. For a vector field $X$ on $M$ we write $X^\top$ and $X^\bot$ for the tangent and normal projections with respect to $P$. 

We use the notation $\nabla_P u$ for the gradient in $P$ of a function $u\in C^1(P)$. When $u$ is defined on an open set of $M$ then $\nabla_Pu=(\nabla u)^\top$. The divergence relative to $P$ of a $C^1$ vector field $X$ on $P$ is the function $(\divv_PX)(p):=\sum_{i=1}^n\,\escpr{D_{e_i}X,e_i}$, where $\{e_1,\ldots,e_n\}$ is any orthonormal basis in the tangent space $T_pP$. The Laplace operator relative to $P$ is the Laplace operator in $P$ with respect to the induced Riemannian metric. It is given by
$\Delta_Pu:=\divv_P(\nabla_Pu)$, for any $u\in C^2(P)$.

If we have a weight $e^h$ in $M$ then its restriction to $P$ produces a structure of weighted manifold. From \eqref{eq:flaplacian} the associated  \emph{$h$-Laplacian} $\Delta^h_P$ has the expression
\begin{equation}
\label{weighted-Laplacian}
\Delta^h_Pu=\Delta_Pu+\escpr{\nabla_P h,\nabla_Pu},
\end{equation}
for any $u\in C^2(P)$. We say that the submanifold $P$ is \emph{$h$-parabolic} when $P$ is weighted parabolic as a weighted manifold. Otherwise we say that $P$ is \emph{$h$-hyperbolic}. By Theorem~\ref{th:parchar} the $h$-parabolicity of $P$ is equivalent to that $\text{Cap}^h_P(K)=0$ for some compact set $K\subeq P$ with non-empty interior in $P$, where $\text{Cap}^h_P$ denotes the \emph{$h$-capacity relative to $P$}.  Clearly a compact submanifold is always $h$-parabolic. In Section~\ref{sec:criteria} we will provide parabolicity criteria for non-compact submanifolds under conditions on their weighted extrinsic geometry. In precise terms, we will assume restrictions on the weighted mean curvature that we now introduce.

Suppose first that $P$ is a two-sided hypersurface, i.e., $P$ admits a smooth unit normal vector field $N$. Following Gromov~\cite[Sect.~9.4.E]{gromov-GAFA} and Bayle~\cite[Sect.~3.4.2]{bayle-thesis}, the \emph{weighted mean curvature} or \emph{$h$-mean curvature} of $P$ is the function
\begin{equation}
\label{eq:fmc}
H^h_P:=(m-1)\,H_P-\escpr{\nabla h,N},
\end{equation}
where $H_P$ is the mean curvature of $P$ in $M$ defined by equality $(m-1)\,H_P:=-\divv_P N$. This notion is related to the isoperimetric problem in $M$ with weight $e^h$. Indeed, the first variational formulas for the weighted volume and area in \eqref{eq:volarea} imply that $H^h_P$ is constant on $P$ if and only if $P$ is a critical point of the weighted area under variations preserving the weighted volume, see \cite[Prop.~3.2]{rcbm} and \cite[Cor.~3.3]{castro-rosales}. 

Next, we consider an arbitrary submanifold $P^n$. Recall that the mean curvature vector of $P$ is the normal vector field $\ovh$ such that $n\ovh:=-\sum_{i=1}^{m-n}\,(\divv_P N_i)\,N_i$, where $\{N_1,\ldots,N_{m-n}\}$ is any local orthonormal basis of vector fields normal to $P$. Thus, it is clear that
\begin{equation}
\label{eq:divnorm}
\divv_P X=-\escpr{n\ovh,X}, \quad\text{for any $C^1$ vector field } X \text{ normal to } P. 
\end{equation}
We define the \emph{$h$-mean curvature vector} of $P$ as the normal vector field to $P$ given by
\begin{equation}
\label{eq:mcvector}
\ovh^h:=n\ovh-(\nabla h)^\bot.
\end{equation}
For a two-sided hypersurface $P$ it is clear from \eqref{eq:fmc} that $\ovh^h=H^h_P\,N$. We say that a submanifold $P$ is \emph{$h$-minimal} if $\ovh^h=0$ on $P$. More generally, a submanifold $P$ has \emph{bounded $h$-mean curvature} if $|\ovh^h|\leq c$ on $P$ for some $c\geq 0$. 

\begin{example}
\label{ex:singularity}
As we pointed out in the Introduction, some weighted minimal submanifolds play an important role in the singularity theory of the mean curvature flow. For instance, in $\rr^m$ with weight $e^{\kappa |p|^2/2}$, $\kappa\in\{-1,1\}$, a weighted minimal submanifold $P^n$ satisfies $n\overline{H}_P(p)=\kappa\,p^\bot$ for any $p\in P$. As it is was shown by Colding and Minicozzi~\cite[Lem.~2.2]{colding-minicozzi}, \cite[Sect.~1.1]{cm2} this is the equation of the self-similar solutions to the mean curvature flow (\emph{self-shrinkers} for $\kappa=-1$ and \emph{self-expanders} for $\kappa=1$). On the other hand, the minimal submanifolds in $\rr^m=\rr^{m-1}\times\rr$ for the weight $e^h$ with $h(x,t):=t$ are those for which $n\overline{H}_P=\ptl_t^\bot$, where $\ptl_t$ is the unit vertical vector field. These submanifolds are called \emph{translating solitons} since their evolution under the mean curvature flow consist of vertical translations. Coming back to the Gaussian weight $e^{-|p|^2/2}$ it is usual to call \emph{$\la$-hypersurfaces} to those hypersurfaces of constant weighted mean curvature $\la$.
\end{example}

\subsection{Model spaces}
\label{subsec:models}
\noindent

Here we introduce the ambient manifolds where the main results of the paper will be established. We recall some facts about their geometry and we characterize their weighted parabolicity with respect to a radial weight. The reader is referred to \cite[Sect.~3]{grigoryan}, \cite[p.~29]{greene-wu} and \cite[Ch.~7]{oneill} for further details.

Consider a complete Riemannian manifold $M^m$ with a pole $o\in M$. This means that the exponential map $\exp_o:T_oM\to M$ is a diffeomorphism (so that $M$ is homeomorphic to $\rr^m$). Hence we have geodesic polar coordinates $(t,\theta)$ in $M\setminus\{o\}$ defined for $t\in\rr^+$ and $\theta\in\sph^{m-1}$. We say that $M$ is a \emph{model space} if the Riemannian metric in $M\setminus\{o\}$ is rotationally symmetric, i.e., its expression with respect to $(t,\theta)$ is $dt^2+w(t)^2\,d\theta^2$, where $w$ is a smooth positive function in $\rr^+$ and $d\theta$ is the standard metric in the unit sphere $\sph^{m-1}$. This implies in particular that $M\setminus\{o\}$ is isometric to the warped product $\rr^+\times_w\,\sph^{m-1}$. The \emph{warping function} $w$ extends to $0$ in such a way that $w(0)=0$, $w'(0)=1$ and $w^{(k)}(0)=0$ for any even derivation order $k$. This function is uniquely determined by $M$; indeed, the Riemannian area of the metric sphere of radius $t>0$ centered at $o$ equals $c_m\,w(t)^{m-1}$, where $c_{m}$ is the Euclidean area of $\sph^{m-1}$. We will denote by $M^m_w$ the $m$-dimensional model space of warping function $w$. Sometimes we will omit the dimension and simply write $M_w$.

\begin{examples}
Many important Riemannian manifolds are model spaces. This is the case of the simply connected Riemannian space forms of non-positive sectional curvature. For the Euclidean space $\rr^m$ the warping function is $w(t)=t$. For the hyperbolic space $\hh^m(\kappa)$ of constant sectional curvature $\kappa<0$ we have $w(t)=\frac{1}{\sqrt{-\kappa}}\,\sinh(\sqrt{-\kappa}\,t)$. Other interesting model spaces are the hypersurfaces of revolution obtained by rotating a curve around a line which meets the curve at only one point. There are also model spaces, like the hyperbolic paraboloid in $\rr^3$ which are neither spaces forms nor hypersurfaces of revolution. 
\end{examples}

In a model space $M^m_w$ we denote by $r:M_w\to\rr$ the \emph{distance function} with respect to $o$. Since $o$ is a pole then $r\in C^\infty(M_w\setminus\{o\})$ and $|\nabla r|=1$. Let $B_t:=r^{-1}([0,t))$ and $S_t:=r^{-1}(t)$ be the open metric ball and the metric sphere of radius $t>0$ centered at $o$. Note that $S_t$ is always a smooth compact connected hypersurface, and that $-\nabla r$ is the unit normal along $S_t$ pointing into $B_t$. 

Most of the geometry of $M_w$ can be described in terms of the functions $w$ and $r$. For example, given any point $p\in M_w\setminus\{o\}$, the sectional curvatures at $p$ with respect to planes containing $(\nabla r)_p$ equal $-w''/w$ evaluated at $r(p)$. On the other hand, the mean curvature $H(t)$ of the sphere $S_t$ with respect to $-\nabla r$ satisfies
\begin{equation}
\label{eq:mcspheres}
H(t)=\frac{w'(t)}{w(t)},
\end{equation}
so that $S_t$ is a constant mean curvature hypersurface. The function $H(t)$ is related to the Hessian of $r$. More precisely, we have this equality for any $p\in M_w\setminus\{o\}$, see \cite[Prop.~2.20]{greene-wu}
\begin{equation}
\label{eq:hessr}
(\text{Hess}\,r)_p(X,Y)=H(r(p))\,\left\{\escpr{X,Y}-\escpr{(\nabla r)_p,X}\,\escpr{(\nabla r)_p,Y}\right\}, \quad X,Y\in T_pM_w.
\end{equation}
By tracing the previous identity we get
\begin{equation}
\label{eq:deltar}
(\Delta r) (p)=(m-1)\,H(r(p)), \quad p\in M_w\setminus\{o\}.
\end{equation}

Now we consider a weight $e^h$ in $M^m_w$. Since the Riemannian metric in $M_w$ is rotationally symmetric it is natural to assume that $e^h$ is a \emph{radial weight}, which means that $h=f\circ r=f(r)$ for some function $f\in C^1(\rr^+_0)$ with $f'(0)=0$. Sometimes a model space with a radial weight is called a \emph{weighted model} \cite[Sect.~2.4]{grigoryan2}. According to \eqref{eq:volarea} and the coarea formula, the weighted area of $S_t$ and the weighted volume of $B_t$ are respectively given by
\begin{equation}
\label{eq:ahsphere}
\begin{split}
A_h(S_t)&=c_m\,w(t)^{m-1}\,e^{f(t)},
\\
V_h(B_t)&=c_m\,\int_0^t w(s)^{m-1}\,e^{f(s)}\,ds.
\end{split}
\end{equation}
From \eqref{eq:fmc} and \eqref{eq:mcspheres}, the $h$-mean curvature of $S_t$ with respect to the unit normal $-\nabla r$ is
\begin{equation}
\label{eq:mchspheres}
H^h(t)=(m-1)\,H(t)+f'(t)=(m-1)\,\frac{w'(t)}{w(t)}+f'(t).
\end{equation}

In the next result we gather some identities that will be helpful in Section~\ref{sec:criteria}. By following \cite[Ex.~4.2]{grigoryan} and \cite[Prop.~6.3]{mp-transience} we compute the weighted capacity of a ball $B_t$ in a weighted model space, by solving the weighted Dirichlet problem \eqref{eq:laplace} for concentric balls centered at the pole.

\begin{proposition}
\label{prop:epm}
In the model space $M^m_w$ we consider a radial weight $e^h$ with $h=f(r)$. Then, for any $\rho, R>0$ such that $\rho<R$, the $h$-capacity potential of the capacitor $(\overline{B}_\rho,B_R)$ is the radial function $u:=\var(r)$, where
\[
\var(s):=\left(\int_\rho^R\frac{dt}{A_h(S_t)}\right)^{-1}\,\left(\int_s^R\frac{dt}{A_h(S_t)} \right), \quad \rho\leq s\leq R.
\]
As a consequence, we have
\[
\emph{Cap}^h(B_\rho,B_R)=|\var'(\rho)|\,A_h(S_\rho)=\left(\int_\rho^R\frac{dt}{A_h(S_t)}\right)^{-1},
\]
and so
\[
\emph{Cap}^h(B_\rho)=\left(\int_\rho^{\infty}\frac{dt}{A_h(S_t)}\right)^{-1}.
\]
\end{proposition}

\begin{proof}
The $h$-capacity potential of $(\overline{B}_\rho,B_R)$ is the unique solution to the Dirichlet problem
\begin{equation}
\label{eq:problem2}
\begin{cases}
\Delta^h u=0\,\,\,&\text{in\, $B_R\setminus\overline{B}_\rho$},\\
\phantom{\Delta^h }u=1\,\,\,&\text{in\, $S_\rho$}, \\
\phantom{\Delta^h }u=0\,\,\,&\text{in\, $S_R$}.
\end{cases}
\end{equation}
Given a radial function $u=\psi(r)$ with $\psi\in C^2([\rho,R])$, the chain rule implies $\nabla u=\psi'(r)\,\nabla r$ and $\Delta u=\psi''(r)+\psi'(r)\,\Delta r$ on $B_R\setminus\overline{B}_\rho$. By taking into account \eqref{eq:flaplacian} and \eqref{eq:deltar}, we obtain
\[
\Delta^h u=\psi''(r)+(m-1)\,H(r)\,\psi'(r)+f'(r)\,\psi'(r)
\]
on $B_R\setminus\overline{B}_\rho$. Hence $u$ is $h$-harmonic if and only if $\psi''+\big((m-1)\,H+f'\big)\,\psi'=0$. From here, a straightforward computation using \eqref{eq:ahsphere} shows that the function $\var$ in the statement provides the unique solution to \eqref{eq:problem2}. Now, the calculus of $\text{Cap}^h(B_\rho,B_R)$ comes from \eqref{eq:capint} having in mind that $\var'<0$ and $\nu=-\nabla r$ on $S_\rho$. Finally, to get $\text{Cap}^h(B_\rho)$ it suffices to apply \eqref{eq:caplim} since $\{B_t\}_{t>0}$ is an exhaustion of $M_w$ by precompact open sets.
\end{proof}

As a direct consequence of Proposition~\ref{prop:epm} and Theorem~\ref{th:parchar} we can characterize the $h$-parabolicity of weighted models by means of an Ahlfors-type criterion.

\begin{corollary}[{\cite[Ex.~9.5]{grigoryan2}}]
\label{cor:ahlfors}
A model manifold $M^m_w$ with a radial weight $e^h$ is $h$-parabolic if and only if there is some $t_0>0$ such that $\int_{t_0}^{\infty}A_h(S_t)^{-1}\,dt=\infty$.
\end{corollary}

\begin{remark}
We can use the corollary to show that the weighted parabolicity not only depends on the manifold, but also on the weight. For instance, the Euclidean plane is parabolic whereas it is hyperbolic with respect to the (anti)Gaussian weight $e^{r^2/2}$. On the other hand, $\rr^m$ is hyperbolic for any $m\geq 3$, whereas it is parabolic for the Gaussian weight $e^{-r^2/2}$.
\end{remark}

\section{Parabolicity and hyperbolicity results for submanifolds}
\label{sec:criteria}

In this section we provide some criteria ensuring weighted parabolicity or hyperbolicity for submanifolds of model spaces under restrictions on the geometry of the model, the ambient weight, and the weighted mean curvature of the submanifold. As in Corollary~\ref{cor:ahlfors} we will deduce our criteria by estimating the weighted capacity of sets in a certain exhaustion of the submanifold.

Let $P^n$ be a submanifold immersed in a model space $M^m_w$. For any $t>0$, the open metric ball $B_t$ in $M_w$ has associated the \emph{extrinsic open ball} in $P$ defined by
\[ 
D_t:=P\cap B_t=\{p\in P\,;\,r(p)<t\}.
\]
If we assume that $P$ is a non-compact submanifold \emph{properly immersed} in $M_w$, then the family $\{D_t\}_{t>0}$ gives an exhaustion of $P$ by precompact open sets. Moreover, since the restriction to $P\setminus\{o\}$ of the distance function $r$ is smooth, we deduce from Sard's theorem that $D_t$ has smooth non-empty boundary $\ptl D_t=\{p\in P\,;\,r(p)=t\}$ for almost any $t>0$. 

To estimate the weighted capacity of an extrinsic ball $D_\rho$ in $P$ we compare the capacity potentials of the extrinsic capacitors $(\overline{D}_\rho,D_R)$ in $P$  with the radial functions obtained by transplanting to $P$, via the distance function $r$, the capacity potentials of the intrinsic capacitors $(\overline{B}_\rho,B_R)$ in a suitable weighted model with the dimension of $P$. In particular, this requires to compute the weighted Laplacian in \eqref{weighted-Laplacian} for radial functions restricted to $P$. This is done in the next result, which extends to arbitrary weights a known formula for the Riemannian case $h=0$. 

\begin{lemma}
\label{equality-laplace-h} 
Let $P^n$ be a submanifold immersed in a model space $M^m_w$ with a weight $e^h$. Suppose that a radial function $v=\psi(r)$ with $\psi\in C^2$ is defined in some open set $D\subeq P\setminus\{o\}$. Then we have the following equality
\[
\Delta_P^h\,v= \big(\psi''(r)-H(r)\,\psi'(r)\big)\,|\nabla_P r|^2+\big(nH(r)+\escpr{\nabla h,\nabla r}+\escpr{\ovh^h,\nabla r}\big)\,\psi'(r),
\]
where $H$ is the function in \eqref{eq:mcspheres} and $\ovh^h$ is the $h$-mean curvature vector in \eqref{eq:mcvector}.
\end{lemma}

\begin{proof}
The result comes from \eqref{weighted-Laplacian} by having in mind the expression for the Hessian of radial functions on submanifolds, see for instance \cite[Sect.~3.2]{Pa2}. We give a proof for the sake of completeness.

Take a point in $D$ and an orthonormal basis $\{N_1,\ldots,N_{m-n}\}$ of normal vectors at that point.
Since $\nabla_P r=(\nabla r)^\top=\nabla r-(\nabla r)^\bot$, we get
\begin{align*}
\Delta_P r&=\divv_P(\nabla r)-\divv_P((\nabla r)^\bot)=\divv (\nabla r)-\sum_{i=1}^{m-n}\escpr{D_{N_i}\nabla r,N_i}+\escpr{n\overline{H}_P,\nabla r}
\\
&=\Delta r-\sum_{i=1}^{m-n}(\text{Hess}\,r)(N_i,N_i)+\escpr{n\overline{H}_P,\nabla r}=nH(r)-H(r)\,|\nabla_Pr|^2+\escpr{n\overline{H}_P,\nabla r},
\end{align*}
where we have used \eqref{eq:divnorm}, \eqref{eq:hessr}, \eqref{eq:deltar} and that $1=|\nabla r|^2=|\nabla_P r|^2+|(\nabla r)^\bot|^2$. On the other hand, from the chain rule and the previous expression for $\Delta_P r$ we infer
\begin{align*} 
\Delta_P\,v&=\divv_P\big(\psi'(r)\,\nabla_P r\big)=\psi'(r)\,\Delta_P r+\psi''(r)\,|\nabla_P r|^2
\\
&=\big(\psi''(r)-H(r)\,\psi'(r)\big)\,|\nabla_P r|^2+\big(nH(r)+\escpr{n\overline{H}_P,\nabla r}\big)\,\psi'(r).
\end{align*}
Finally, we have
\begin{align*}
\Delta^h_P\,v&=\Delta_P\,v+\escpr{\nabla_P h,\nabla_P v}=\Delta_P\,v+\escpr{\nabla_P h,\nabla v}
\\
&=\Delta_P\,v+\escpr{\nabla_P h,\nabla r}\,\psi'(r)=\Delta_P\,v+\escpr{\nabla h,\nabla r}\,\psi'(r)-\escpr{(\nabla h)^\bot,\nabla r}\,\psi'(r),
\end{align*}
and the claim follows from the expression of $\Delta_P\,v$ and the equality $\ovh^h=n\overline{H}_P-(\nabla h)^\bot$.
\end{proof}

The previous lemma shows that, in order to eventually control the weighted Laplacian of radial functions restricted to $P$, we need to control the functions $\escpr{\nabla h,\nabla r}$ and $\escpr{\ovh^h,\nabla r}$, which measure the radial derivative of the logarithm of the weight, and the radial component of the $h$-mean curvature vector, respectively. Indeed, estimating these quantities by means of radial functions leads to a suitable weighted model to establish our comparisons. 

Now, we are ready to state and prove the main results of this section. The first one is a parabolicity criterion, which in the unweighted case $h=0$ follows from a more general statement by Esteve and the second author~\cite[Thm.~3.4]{esteve-palmer}. 

\begin{theorem}
\label{parabolicity-sub-model} 
Let $P^n$ be a non-compact submanifold properly immersed in a model space $M^m_w$ with weight $e^h$. Suppose that there is a number $t_0>0$ and a continuous function $\alpha:[t_0,\infty)\to\rr$ such that the following inequalities hold in $P\setminus D_{t_0}$
\begin{itemize}
\item[(A)] $\escpr{\nabla h,\nabla r}+\escpr{\ovh^h,\nabla r}\leq\alpha(r)$,
\item[(B)] $nH(r)+\alpha(r)\leq 0$.
\end{itemize}
In the $n$-dimensional model space $M^n_w$, we consider a radial weight $e^{f(r)}$ such that $f(t):=\int_{t_0}^t\alpha(s)\,ds$ for any $t\geq t_0$. Then, for any $\rho\geq t_0$ such that $\ptl D_\rho$ is smooth, we have
\[
\frac{\emph{Cap}^h_P(D_\rho)}{A_h(\ptl D_\rho)}\leq\frac{\C^f(B^n_\rho)}{A_f(S^{n-1}_\rho)},
\]
where $A_h$ and $A_f$ denote the weighted areas in $P$ and $M^n_w$, respectively, $B^n_t$ stands for the open metric ball of radius $t>0$ centered at the pole in $M^n_w$, and $S^{n-1}_t:=\ptl B^n_t$. Moreover, if 
\begin{equation}
\label{eq:icI}
\int_{t_0}^{\infty}\frac{dt}{A_f(S^{n-1}_t)}=\infty,
\end{equation}
then $P$ is $h$-parabolic.
\end{theorem}

\begin{proof}
Fix numbers $\rho\geq t_0$ and $R>\rho$ such that the extrinsic open balls $D_\rho$ and $D_R$ have smooth boundaries. Let $u$ be the $h$-capacity potential of $(\overline{D}_\rho,D_R)$, i.e., the unique solution to the problem
\[
\begin{cases}
\Delta_P^h u=0\,\,\,&\text{in\, $D_R\setminus\overline{D}_\rho$},\\
\phantom{\Delta_P^h }u=1\,\,\,&\text{in\, $\ptl D_\rho$}, \\
\phantom{\Delta_P^h }u=0\,\,\,&\text{in\, $\ptl D_R$}.
\end{cases}
\]
On the other hand, the $f$-capacity potential of the capacitor $(\overline{B^n_\rho},B^n_R)$ in the model space $M^n_w$ is, by Proposition~\ref{prop:epm}, the radial function associated to 
\[
\var(s):=\left(\int_\rho^R\frac{dt}{A_f(S^{n-1}_t)}\right)^{-1}\,\left(\int_s^R\frac{dt}{A_f(S^{n-1}_t)}\right),\quad\rho\leq s\leq R.
\]
As we showed in the proof of Proposition~\ref{prop:epm} the function $\var$ satisfies
\[
\var''+\big((n-1)H+\alpha\big)\,\var'=0,
\]
and so
\[
\var''-H\var'=-\big(nH+\alpha\big)\,\var',
\]
which is a nonpositive function by hypothesis (B) since $\var'<0$.

Now, in the extrinsic annulus $\overline{D}_R\setminus D_\rho$ we define the radial function $v:=\var(r)$. Clearly $v=1$ on $\ptl D_\rho$ and $v=0$ on $\ptl D_R$. By applying Lemma~\ref{equality-laplace-h}, the inequality $|\nabla _P r|\leq 1$ and hypothesis (A), we obtain
\begin{align*}
\Delta^h_P\,v&=\big(\var''(r)-H(r)\,\var'(r)\big)\,|\nabla_P r|^2+\big(nH(r)+\escpr{\ovh^h,\nabla r}+\escpr{\nabla h,\nabla r}\big)\,\var'(r)
\\
&\geq\var''(r)+\big((n-1)\,H(r)+\alpha(r)\big)\,\var'(r)=0.
\end{align*}
Thus, the function $v-u$ is $h$-subharmonic in $D_R\setminus\overline{D}_\rho$ and vanishes along $\ptl (D_R\setminus\overline{D}_\rho)$. As a consequence of the maximum principle in Theorem~\ref{th:mp} we get that, either $v-u=0$ in $D_R\setminus \overline{D}_\rho$, or $v-u$ achieves its maximum along $\ptl (D_R\setminus\overline{D}_\rho)$. From the Hopf boundary point lemma, the latter implies that $\frac{\ptl u}{\ptl\nu}<\frac{\ptl v}{\ptl\nu}$ in $\ptl D_\rho$, where $\nu$ is the outer unit normal along $\ptl (D_R\setminus\overline{D}_\rho)$, which coincides with the unit normal along $\ptl D_\rho$ pointing into $D_\rho$. By taking into account \eqref{eq:capint}, we deduce
\begin{align*}
\text{Cap}^h_P(D_\rho,D_R)&=\int_{\ptl D_\rho}\frac{\ptl u}{\ptl\nu}\,da_h\leq\int_{\ptl D_\rho}\frac{\ptl v}{\ptl\nu}\,da_h\leq\int_{\ptl D_\rho}|\nabla_P v|\,da_h
\\
&=|\var'(\rho)|\,\int_{\ptl D_\rho}|\nabla_Pr|\,da_h\leq |\var'(\rho)|\,A_h(\ptl D_\rho)=\frac{\text{Cap}^f(B^n_\rho,B^n_R)}{A_f(S^{n-1}_\rho)}\,A_h(\ptl D_\rho),
\end{align*}
where we have used the second equation in Proposition~\ref{prop:epm}. This shows that inequality
\[
\frac{\text{Cap}^h_P(D_\rho,D_R)}{A_h(\ptl D_\rho)}\leq\frac{\text{Cap}^f(B^n_\rho,B^n_R)}{A_f(S^{n-1}_\rho)}
\]
holds for a dense set of numbers $R>\rho$. By taking limits when $R\to\infty$ the desired comparison follows from \eqref{eq:caplim}. Finally the integrability condition in \eqref{eq:icI} is equivalent, by Corollary~\ref{cor:ahlfors}, to that the model space $M^n_w$ is $f$-parabolic. Hence, Theorem~\ref{th:parchar} and the capacity comparison yield that $\text{Cap}^h_P(D_\rho)=0$ for some $\rho\geq t_0$. From this we conclude that $P$ is $h$-parabolic. 
\end{proof}

Our second result provides weighted hyperbolicity by reversing the hypotheses of Theorem~\ref{parabolicity-sub-model}. For $h=0$ the criterion below is consequence of a more general result by Markvorsen and the second author, see \cite[Thm.~A, Thm.~7.1]{mp-transience}.

\begin{theorem}
\label{hyperbolicity-sub-model} 
Let $P^n$ be a non-compact submanifold properly immersed in a model space $M^m_w$ with weight $e^h$. Suppose that there is a number $t_0>0$ and a continuous function $\alpha:[t_0,\infty)\to\rr$ such that the following inequalities hold in $P\setminus D_{t_0}$
\begin{itemize}
\item[(A)] $\escpr{\nabla h,\nabla r}+\escpr{\ovh^h,\nabla r}\geq\alpha(r)$,
\item[(B)] $nH(r)+\alpha(r)\geq 0$.
\end{itemize}
In the $n$-dimensional model space $M^n_w$, we consider a radial weight $e^{f(r)}$ such that $f(t):=\int_{t_0}^t\alpha(s)\,ds$ for any $t\geq t_0$. Then, for any $\rho\geq t_0$ such that $\ptl D_\rho$ is smooth, we have
\[
\emph{Cap}^h_P(D_\rho)\geq\frac{\C^f(B^n_\rho)}{A_f(S^{n-1}_\rho)}\,\int_{\partial D_\rho} |\nabla_P r|\,da_h,
\]
where $A_f$ denotes the weighted area in $M^n_w$, $B^n_t$ stands for the open metric ball of radius $t>0$ centered at the pole in $M^n_w$, and $S^{n-1}_t:=\ptl B^n_t$. Moreover, if 
\begin{equation}
\label{eq:icII}
\int_{t_0}^{\infty}\frac{dt}{A_f(S^{n-1}_t)}<\infty,
\end{equation}
then $P$ is $h$-hyperbolic.
\end{theorem}

\begin{proof} 
Take $\rho\geq t_0$ and $R>\rho$ such that $D_\rho$ and $D_R$ have smooth boundaries. By using Sard's theorem we can suppose that $\nabla_Pr\neq 0$ along $\ptl D_\rho$. We proceed as in the proof of Theorem \ref{parabolicity-sub-model}. Following the notation there, and reversing all the inequalities, we deduce that $u-v$ is an $h$-subharmonic function on $D_R\setminus\overline{D}_\rho$ vanishing at the boundary. The maximum principle and the Hopf boundary point lemma imply that $\frac{\ptl u}{\ptl\nu}\geq\frac{\ptl v}{\ptl\nu}$ along $\ptl D_\rho$. Moreover, $\nu=\frac{\nabla_P v}{|\nabla_P v|}$ along $\ptl D_\rho$ since $v=1$ and $\nabla_P v=\var'(\rho)\,\nabla_P r\neq 0$ along $\ptl D_\rho$.
As a consequence
\begin{align*}
\text{Cap}^h_P(D_\rho,D_R)&=\int_{\ptl D_\rho}\frac{\ptl u}{\ptl\nu}\,da_h\geq\int_{\ptl D_\rho}|\nabla_P v|\,da_h=|\var'(\rho)|\,\int_{\ptl D_\rho}|\nabla_Pr|\,da_h
\\
&=\frac{\text{Cap}^f(B^n_\rho,B^n_R)}{A_f(S^{n-1}_\rho)}\,\int_{\ptl D_\rho}|\nabla_Pr|\,da_h,
\end{align*}
so we get the desired comparison by letting $R\to\infty$. On the other hand, the integrability condition in \eqref{eq:icII} is equivalent, by Corollary~\ref{cor:ahlfors}, to the $f$-hyperbolicity of the model space $M^n_w$. From Theorem~\ref{th:parchar} and the previous comparison we infer that $\text{Cap}^h_P(D_\rho)>0$. This shows that $P$ is $h$-hyperbolic and completes the proof.
\end{proof}

Let us make some comments about the different hypotheses in the previous theorems.

\begin{remarks}
\label{re:hypotheses}
1. The hypothesis (A) holds provided $\escpr{\nabla h,\nabla r}$ and $\escpr{\ovh^h,\nabla r}$ are bounded (from above or from below) by continuous radial functions. By the Cauchy-Schwarz inequality this is guaranteed, for instance, if $h$ is radial and $P$ has bounded $h$-mean curvature. In the Riemannian context, several parabolicity results for submanifolds have been derived under the hypothesis that the radial mean curvature is bounded; besides the aforementioned references \cite{esteve-palmer} and \cite{mp-GAFA}, we refer the reader to \cite{MP5}, \cite{hp-ppar} and \cite{hmp-2009}.

2. The hypothesis (B) means that the function $\alpha(t)$ is balanced with respect to the warping function $w(t)$. This condition has a geometric meaning. In the model space $M^{n+1}_w$ consider a radial weight $e^{f(r)}$ such that $f(t):=\int_{t_0}^t\alpha(s)\,ds$ for any $t\geq t_0$. Then, equation \eqref{eq:mchspheres} shows that the $f$-mean curvature of the metric sphere or radius $t$ centered at the pole equals $nH(t)+\alpha(t)$ for any $t\geq t_0$. Hence, hypothesis (B) may be seen as a kind of weighted mean convexity for such spheres.
 
3. As we have mentioned in the proofs, the integrability conditions \eqref{eq:icI} and \eqref{eq:icII} are equivalent by Corollary~\ref{cor:ahlfors} to the $f$-parabolicity or $f$-hyperbolicity of the corresponding weighted comparison model.
\end{remarks}

The integrability hypotheses \eqref{eq:icI} and \eqref{eq:icII} are in general difficult to check. In the following remarks we show some sufficient conditions for them.

\begin{remarks}
\label{re:integrability}
1. The balance conditions (B) in Theorems~\ref{parabolicity-sub-model} and \ref{hyperbolicity-sub-model} imply the integrability conditions \eqref{eq:icI} and \eqref{eq:icII} under further hypotheses. Suppose for instance that $nH(r)+\alpha(r)\leq 0$ on $P\setminus D_{t_0}$. By integrating this inequality and taking into account \eqref{eq:mcspheres}, we obtain
\[
f(t):=\int_{t_0}^t\alpha(s)\,ds\leq\log\left(\frac{w(t_0)^n}{w(t)^n}\right),\quad t\geq t_0.
\]
From here and \eqref{eq:ahsphere} we get
\[
\int_{t_0}^\infty\frac{dt}{A_f(S^{n-1}_t)}=c^{-1}_n\,\int_{t_0}^\infty w(t)^{1-n}\,e^{-f(t)}\,dt\geq\frac{c_n^{-1}}{w(t_0)^n}\int_{t_0}^\infty w(t)\,dt.
\]
Thus, if $w\notin L^1(0,\infty)$, then $\int_{t_0}^\infty A_f(S^{n-1}_t)^{-1}\,dt=\infty$. The same argument shows that, if $w\in L^1(0,\infty)$, then the condition $nH(r)+\alpha(r)\geq 0$ on $P\setminus D_{t_0}$ implies that $\int_{t_0}^\infty A_f(S^{n-1}_t)^{-1}\,dt<\infty$. 

2. Suppose that $w(t)\to L$ with $L\in(0,\infty]$ when $t\to\infty$. Then, we have
\[
\lim_{t\to\infty}\frac{w(t)^{1-n}\,e^{-f(t)}}{e^{-f(t)}}=\lim_{t\to\infty} w(t)^{1-n}=L'\geq 0.
\]
It follows that the integral $\int_{t_0}^\infty w(t)^{1-n}\,e^{-f(t)}\,dt$, which depends on the warping function $w$, is finite provided the integral $\int_{t_0}^\infty e^{-f(t)}\,dt$, which does not depend on $w$, is finite. In a similar way we deduce that, if $f(t)\to L$ with $L\in(-\infty,\infty]$ when $t\to\infty$, and $\int_{t_0}^tw(t)^{1-n}\,dt<\infty$, then $\int_{t_0}^t w(t)^{1-n}\,e^{-f(t)}\,dt<\infty$.

3. The integrability of $w$ is somehow related to the Riemannian volume of $M^m_w$, which is given by $V(M_w):=c_m\,\int_0^\infty w(t)^{m-1}\,dt$. For instance, if $w$ has a finite limit at infinity, then the condition $w\in L^1(0,\infty)$ yields $V(M_w)<\infty$. On the other hand, if $w$ tends to $L\in(0,\infty]$ at infinity, then $w\notin L^1(0,\infty)$, the Riemannian areas of the metric spheres $S_t$ tend to $c_m\,L^{n-1}$ and $V(M_w)=\infty$. This happens in Euclidean space $\rr^m$, in hyperbolic space $\mathbb{H}^m(\kappa)$, and in hypersurfaces of revolution in $\rr^{m+1}$ whose distance with respect to the axis of revolution is nondecreasing.
\end{remarks}

In the remainder of this section we will deduce some consequences of the previous theorems for submanifolds having bounded $h$-mean curvature with respect to suitable weights. A very useful criterion is the following.

\begin{corollary}
\label{cor:useful}
Let $M^m_w$ be a model space such that $w\notin L^1(0,\infty)$ $($resp. $w\in L^1(0,\infty)$$)$ and the function $H$ in \eqref{eq:mcspheres} is bounded at infinity. Consider a weight $e^h$ and a non-compact submanifold $P^n$ properly immersed in $M_w$ such that $P$ has bounded $h$-mean curvature and $\escpr{\nabla h,\nabla r}\leq\beta(r)$ $($resp. $\escpr{\nabla h,\nabla r}\geq\beta(r)$$)$ on $P\setminus\{o\}$, for some continuous function $\beta:\rr^+\to\rr$ with $\beta(t)\to -\infty$ $($resp. $\beta(t)\to\infty$$)$ when $t\to\infty$. Then, $P$ is $h$-parabolic $($resp. $h$-hyperbolic$)$.
\end{corollary}

\begin{proof}
We only prove the parabolicity case (the other one is similar). We will apply Theorem~\ref{parabolicity-sub-model}. Take a constant $c\geq 0$ such that $|\ovh^h|\leq c$ on $P$. Then, the Cauchy-Schwarz inequality implies the following estimate in $P\setminus\{o\}$
\[ 
\escpr{\nabla h,\nabla r}+\escpr{\ovh^h,\nabla r}\leq\beta(r)+|\ovh^h|\leq\beta(r)+c,
\]
so that the condition (A) holds. On the other hand, since $H$ is bounded at infinity and $\beta(t)\to-\infty$ when $t\to\infty$, there is $t_0>0$ such that
\[
nH(t)+\beta(t)+c\leq 0, \quad t\geq t_0,
\]
and so, the condition (B) is satisfied. Moreover in Remarks~\ref{re:integrability} we showed that the integrability condition \eqref{eq:icI} comes from (B) since $w\notin L^1(0,\infty)$. We conclude that $P$ is $h$-parabolic, as we claimed.
\end{proof}

\begin{remark}
The hypotheses on $w$ and $H$ in Corollary~\ref{cor:useful} guarantee the balance conditions and hence the integrability conditions \eqref{eq:icI} and \eqref{eq:icII}. These hypotheses are related to geometric quantities in the model space $M^m_w$, like the Riemannian volume and the mean curvature of the metric spheres $S_t$. The fact that $w\notin L^1(0,\infty)$ holds for instance if $w(t)\to L$ with $L\in (0,\infty]$ when $t\to\infty$. This happens for instance in the space forms $\rr^m$ and $\mathbb{H}^m(\kappa)$, where the function $H$ is also bounded at infinity. There are also many examples where $w\in L^1(0,\infty)$ and $H$ is bounded at infinity. This happens when $w(t)=(1+t^k)^{-1}$ with $k>1$, $w(t)=t^k$ with $k<-1$, or $w(t)=e^{-t}$ for $t\geq t_0$.
\end{remark}

As a direct application of Corollary~\ref{cor:useful} we deduce the following result for perturbations of radial weights that will be useful in Section~\ref{sec:char}.

\begin{corollary}
\label{cor:parabolicity}
Let $M^m_w$ be a model space such that $w\notin L^1(0,\infty)$ $($resp. $w\in L^1(0,\infty)$$)$ and the function $H$ in \eqref{eq:mcspheres} is bounded at infinity. Consider a weight $e^h$ with $h:=f(r)+g$, where $f(r),g\in C^1(M_w)$ and $f'(t)\to -\infty$ $($resp. $f'(t)\to\infty$$)$ when $t\to\infty$. If $P^n$ is a non-compact submanifold properly immersed in $M_w$ such that $P$ has bounded $h$-mean curvature and $\escpr{\nabla g,\nabla r}\leq \delta$ $($resp. $\escpr{\nabla g,\nabla r}\geq\delta$$)$ on $P\setminus\{o\}$ for some $\delta\in\rr$, then $P$ is $h$-parabolic $($resp. $h$-hyperbolic$)$.
\end{corollary}

In the particular situation of radial weights we can prove the next corollary.

\begin{corollary}
\label{cor:radialcase}
Let $M^m_w$ be a model space such that $w\notin L^1(0,\infty)$ $($resp. $w\in L^1(0,\infty)$$)$ and the function $H$ in \eqref{eq:mcspheres} is bounded at infinity. Consider a $C^1$ weight $e^{h}$ with $h:=f(r)$ and $f'(t)\to -\infty$ $($resp. $f'(t)\to\infty$$)$ when $t\to\infty$. If $P^n$ is a non-compact submanifold properly immersed in $M_w$ with $|\ovh^h|\leq c$ for some constant $c\geq 0$, then $P$ is $h$-parabolic $($resp. $h$-hyperbolic$)$. Moreover, if instead of assuming $w\in L^1(0,\infty)$, we suppose that $w(t)\to L$ with $L\in (0,\infty]$ when $t\to\infty$, and that
\[
\int_{0}^\infty e^{ct-f(t)}\,dt<\infty,
\] 
then $P$ is $h$-hyperbolic.
\end{corollary}

\begin{proof}
The case where $w\notin L^1(0,\infty)$ (resp. $w\in L^1(0,\infty)$) comes from Corollary~\ref{cor:parabolicity}. For the other integrability condition the statement follows from Theorem~\ref{hyperbolicity-sub-model} and Remarks~\ref{re:integrability}.
\end{proof}

\begin{example}
Take a model space $M^m_w$ with $w\notin L^1(0,\infty)$ and $H$ bounded at infinity. In $M_w$ we consider the radial weight $e^{f(r)}$ such that  $f(t):=a\,t^k+g(t)$, where $a<0$, $k>1$ and $g:\rr^+_0\to\rr$ is a $C^2$ concave function with $g'(0)=0$. Note that $g'\leq 0$ since $g'$ is nonincreasing and $g'(0)=0$. Thus $f'(t)\to-\infty$ when $t\to\infty$, and we can apply Corollary~\ref{cor:radialcase} to deduce that any submanifold properly immersed in $M_w$ with bounded $f$-mean curvature is $f$-parabolic. This holds in particular for radial log-concave perturbations of the Gaussian weight $e^{-r^2/2}$ in $M_w$. In the special case of proper self-shrinkers (i.e. minimal submanifolds for the Gaussian weight in $\rr^m$), the parabolicity was obtained by Cheng and Zhou~\cite[Thm.~4.1]{cheng-zhou-volume}, who proved that they have finite weighted volume. 
\end{example}

\begin{example}
Consider a model space $M^m_w$ with $H$ bounded at infinity. Take a radial weight $e^{f(r)}$ such that $f(t):=a\,t^k+g(t)$, where $a>0$, $k>1$ and $g:\rr^+_0\to\rr$ is a $C^2$ convex function with $g'(0)=0$. Suppose that either $w\in L^1(0,\infty)$, or $w(t)\to L$ with $L\in(0,\infty]$ when $t\to\infty$. For any $c\geq0$, we get
\[
\lim_{t\to\infty}\frac{e^{ct-f(t)}}{e^{ct-at^k}}=\lim_{t\to\infty}e^{-g(t)}=L'\geq 0,
\]
so that $\int_0^\infty e^{ct-f(t)}\,dt<\infty$. Thus, Corollary~\ref{cor:radialcase} entails that any non-compact submanifold $P^n$ properly immersed in $M_w$ with bounded $f$-mean curvature is $f$-hyperbolic.
This holds in particular for radial log-convex perturbations of the (anti)Gaussian weight $e^{r^2/2}$ in $M_w$.
\end{example}

Recall that minimal submanifolds in $\rr^m$ for the (anti)Gaussian weight $e^{r^2/2}$ coincide with the self-expanders of the mean curvature flow. Since they are always noncompact, see \cite[Prop.~5.3]{cao-li}, we can deduce the following consequence from the previous example. 

\begin{corollary} 
\label{cor:self-expanders}
A non-compact submanifold properly immersed in $\rr^m$ with bounded mean curvature with respect to the $($anti$)$Gaussian weight $e^{r^2/2}$ is weighted hyperbolic. In particular, all properly immersed self-expanders are weighted hyperbolic.
\end{corollary}

\begin{remark}
Recently Gimeno and the second author \cite[Cor.~3.3]{gimeno-palmer-preprint} have proved that any complete and possibly non proper self-expander $P^n$ of the mean curvature flow in $\rr^m$ with dimension $n>2$ is hyperbolic in the classical (Riemannian) sense. We must mention that there are $2$-dimensional self-expanders which are parabolic in the classical sense, see \cite[Ex.~5.2]{gimeno-palmer-preprint}.
\end{remark}

There are radial weights where Corollary~\ref{cor:radialcase} does not apply. The next result covers an interesting case, which includes the radial homogeneous weights in $\rr^m$.

\begin{corollary}
\label{cor:radial2}
Consider a model space $M^m_w$ where the spheres $S_t$ are convex at infinity, i.e., the function $H$ in \eqref{eq:mcspheres} satisfies $H(t)\geq 0$ for any $t\geq t_0$. Let $P^n$ be an $h$-minimal non-compact submanifold properly immersed in $M_w$ with $o\notin P$. Then, for the weight $w(r)^k$ in $M_w\setminus\{0\}$, we have:
\begin{itemize}
\item[(i)] if $k\leq -n$, then $P$ is weighted parabolic,
\item[(ii)] if $k>-n$ and $\int_{t_0}^\infty w(t)^{1-n-k}\,dt<\infty$, then $P$ is weighted hyperbolic.
\end{itemize}
\end{corollary} 

\begin{proof}
The statement comes from Theorems~\ref{parabolicity-sub-model} and \ref{hyperbolicity-sub-model}. The logarithm of the weight is the function $h=f(r)$ with $f(t):=k\,\log(w(t))$. It is clear that
\[
\escpr{\nabla h,\nabla r}+\escpr{\ovh^h,\nabla r}=k\,H(r),
\]
and so condition (A) holds. Moreover, since $H(t)\geq 0$ for any $t\geq t_0$, the condition (B) also holds. The fact that $H(t)\geq 0$ for $t\geq t_0$ ensures that $w$ is nondecreasing, so that $w\notin L^1(0,\infty)$. By Remarks~\ref{re:integrability} it follows that the integrability condition \eqref{eq:icI} is satisfied. On the other hand, the integrability hypothesis in (ii) coincides with \eqref{eq:icII}.
\end{proof}

More consequences of Theorems~\ref{parabolicity-sub-model} and \ref{hyperbolicity-sub-model}, including some geometric properties for generalized translating solitons in $\rr^m$, will be derived later in the context of our characterization results for submanifolds.

\section{Characterization results for submanifolds}
\label{sec:char}

The Liouville property of parabolic submanifolds provides interesting information when applied to functions with geometric meaning. In this section we obtain several results in this line by combining our previous study of parabolicity with the analysis of some functions having spheres, cylinders or hyperplanes as level sets. As a consequence, for certain weights in a model space $M^m_w$, we will be able to characterize hypersurfaces with bounded weighted mean curvature and contained into some regions of $M_w$. We will focus on Euclidean space $\rr^m$, where we will derive half-space and Bernstein-type theorems. 

The section is organized into several subsections where we treat the different situations.

\subsection{Ball results}
\label{subsec:ball}
\noindent

Here we study submanifolds inside or outside a metric ball centered at the pole in $M^m_w$. For a radial weight $e^h$ in $M_w$ with $h=f(r)$ recall that the $h$-mean curvature of the metric sphere $S_t$ with respect to the unit normal $-\nabla r$ is given by 
\begin{equation}
\label{eq:mchspheres2}
H^h(t)=(m-1)\,H(t)+f'(t),
\end{equation}
where $H$ is the Riemannian mean curvature of $S_t$ in \eqref{eq:mcspheres}. This shows that $S_t$ has constant $h$-mean curvature. Note that $H^h(t)\to\infty$ when $t\to 0$. Moreover, if $H$ is bounded at infinity and $f'(t)\to-\infty$ when $t\to\infty$, then $H^h(t)\to-\infty$ when $t\to\infty$. This would imply that, for any $\la_0\in\rr$, there is $t_0>0$ such that $H^h(t_0)=\la_0$. In particular, there would exist $h$-minimal spheres in $M_w$. 

\begin{example}
\label{ex:critical}
In $\rr^m$ with Gaussian weight $e^{-r^2/2}$ we have $H^h(t)=(m-1)/t-t$, so that the critical radius $t_0>0$ for which $H^h(t_0)=\la_0$ is
\[
t_0=\frac{-\la_0+\sqrt{\la_0^2+4\,(m-1)}}{2}.
\] 
Hence the unique $h$-minimal sphere is $S_{\sqrt{m-1}}$.
\end{example}

Now, we can prove a geometric restriction for $n$-dimensional submanifolds outside an open ball $B_t$. We need to control the $h$-mean curvature function of the metric spheres inside the model space $M^{n+1}_w$ with radial weight $e^{h}$. This function is denoted by $H^h_n(t)$, and it is defined as in \eqref{eq:mchspheres2} by replacing $m-1$ with $n$. Obviously $H^h_{m-1}=H^h$. 

\begin{theorem}
\label{th:ball1}
Let $M^m_w$ be a model space such that $w\notin L^1(0,\infty)$ and the function $H$ in \eqref{eq:mcspheres} is bounded at infinity. Consider a weight $e^h$ with $h:=f(r)$ and $f'(t)\to-\infty$ when $t\to\infty$. For a fixed number $\la_0\geq 0$, let $t_0>0$ be the first number such that $H_n^h(t)\leq-\la_0$ for any $t\geq t_0$. If $P^n$ is a submanifold properly immersed in $M_w$ with $P\sub M_w\setminus B_{t_0}$ and $|\ovh^h|\leq\la_0$, then $P\subeq S_t$ for some $t\geq t_0$ such that $H_n^h(t)=-\la_0$, and $\ovh^h=\la_0\,\nabla r$ on $P$. If $P$ is a hypersurface then $P=S_t$. Moreover, if $H_n^h$ is decreasing then $P\subeq S_{t_0}$, with equality when $P$ is a hypersurface.  
\end{theorem}

\begin{proof}
A submanifold $P$ in the conditions of the statement is $h$-parabolic. This is clear if $P$ is compact; otherwise it comes from Corollary~\ref{cor:radialcase}. Let $\psi:\rr^+_0\to\rr$ be the function $\psi(t):=\int_0^t w(s)\,ds$. Obviously $\psi\in C^\infty(\rr^+_0)$ and $\psi'(0)=0$. This implies that the function $v:=\psi(r)$ is $C^2$ on $M_w$. Note that $\psi''-H\,\psi'=0$ in $\rr^+$ by \eqref{eq:mcspheres}. From Lemma~\ref{equality-laplace-h} and equation~\eqref{eq:mchspheres2}, we get
\[
\Delta^h_P\,v=\big(H_n^h(r)+\escpr{\ovh^h,\nabla r}\big)\,w(r).
\]
By using the Cauchy-Schwarz inequality and our hypotheses, we deduce that 
\begin{equation}
\label{eq:hoyta}
\Delta^h_P\,v\leq\big(|\ovh^h|-\la_0\big)\,w(r)\leq 0.
\end{equation}
On the other hand $v\geq \psi(t_0)$ on $P$ since $\psi$ is an increasing function. From the $h$-parabolicity of $P$ it follows that $v$ is constant on $P$ and so, $P\subeq S_t$ for some $t\geq t_0$. Since $\Delta^h_P\,v=0$, then all the inequalities in \eqref{eq:hoyta} must be equalities. Hence $H^h_n(t)=-\la_0$ and $\ovh^h=\la_0\,\nabla r$ on $P$. As a consequence $t=t_0$ if $H_n^h$ is decreasing. Finally, when $n=m-1$ we obtain $P=S_t$ since $P$ is properly immersed and $S_t$ is connected, see Remark~\ref{re:details} below. This completes the proof.
\end{proof}

\begin{remark}
In the special case where $H(t)\geq 0$ for any $t>0$, it is clear that $H^h_n(t)\leq H^h(t)$ and so, the hypothesis $H^h(t)\leq-\lambda_0$ for any $t\geq t_0$ gives $H_n^h(t)\leq-\la_0$ for any $t\geq t_0$.
\end{remark}

\begin{remark}
\label{re:details}
Let $P^n$ be a submanifold properly immersed in a model space $M^m_w$. If  $P\subeq P'$ for some connected $n$-dimensional submanifold $P'$ of $M_w$, then $P=P'$. To see this, note that $P$ is an open subset of $P'$ since it is an immersed submanifold of $P'$ with the same dimension. On the other hand, the properness of $P$ and the completeness of $M_w$ guarantee that $P$ is a closed subset of $M_w$. Hence, the connectivity of $P'$ allows to conclude that $P=P'$.
\end{remark}

For submanifolds inside a closed ball $\overline{B}_t$ we can deduce a similar result. In this case the properness of the submanifold implies compactness, so that we do not need assumptions on the model space $M^m_w$ nor on the weight $e^h$ to ensure parabolicity. By following the proof of the previous theorem we obtain this statement.

\begin{theorem}
\label{th:ball2}
Consider a radial weight $e^h$ in a model space $M^m_w$. For any number $\la_0\geq 0$ we choose $t_0>0$ such that $H_n^h(t)\geq\la_0$ for any $t\leq t_0$. If $P^n$ is a submanifold properly immersed in $M_w$ with $P\sub\overline{B}_{t_0}$ and $|\ovh^h|\leq\la_0$, then $P\subeq S_t$ for some $t\leq t_0$ such that $H_n^h(t)=\la_0$, and $\ovh^h=-\la_0\,\nabla r$ on $P$. If $P$ is a hypersurface then $P=S_t$. Moreover, if $H_n^h$ is decreasing and $t_0>0$ is the unique number for which $H_n^h(t_0)=\la_0$, then $P\subeq S_{t_0}$, and we have equality $P=S_{t_0}$ when $P$ is a hypersurface. 
\end{theorem}

The previous theorems lead to the following corollary for $h$-minimal submanifolds which is interesting in itself.

\begin{corollary}
\label{cor:ball}
Let $M^m_w$ be a model space such that $w\notin L^1(0,\infty)$ and $H$ is bounded at infinity. Consider a weight $e^h$ with $h:=f(r)$ and $f'(t)\to-\infty$ when $t\to\infty$. Suppose that $H_n^h$ is decreasing, and let $t_0>0$ be the unique number such that $H^h_n(t_0)=0$. If $P^n$ is an $h$-minimal submanifold properly immersed in $M_w$ with $P\sub\overline{B}_{t_0}$ or $P\sub M_w\setminus B_{t_0}$, then $P\subeq S_{t_0}$. Moreover, in the case $n=m-1$ then $P=S_{t_0}$.  
\end{corollary}

\begin{example}
The previous results apply in $\rr^m$, $\mathbb{H}^m(\kappa)$, and convex paraboloids of revolution in $\rr^{m+1}$, with a weight $e^{f(r)}$ such that $f$ is concave and $f'(t)\to-\infty$ when $t\to\infty$. This is the case of $f(t):=a\,t^k$ where $a<0$ and $k>1$, which includes the Gaussian weight $e^{-r^2/2}$. From the calculus in Example~\ref{ex:critical} the critical ball for $n$-dimensional self-shrinkers in $\rr^m$ is $B_{\sqrt{n}}$. We remark that properly immersed self-shrinker hypersurfaces inside some Euclidean balls were described by Vieira and Zhou~\cite[Thm.~1]{vieira-zhou}. Geometric restrictions for complete self-shrinker hypersurfaces in a ball $\overline{B}_t\sub\rr^{m}$ were given by Pigola and Rimoldi~\cite[Thm.~1]{pigola-rimoldi} by controlling the second fundamental form. Recently Gimeno and the second author~\cite[Thm.~6.1, Cor.~6.2]{gimeno-palmer-preprint} have obtained Theorem~\ref{th:ball2} for a complete self-shrinker $P$ of any codimension which is parabolic in the classical (Riemannian) sense. As an interesting fact, they have also shown that $P$ is contained in the corresponding sphere $S_t$ as a minimal submanifold.
\end{example}

Theorem~\ref{th:ball2} also implies non-existence of compact minimal submanifolds with respect to some radial weights. In this direction we get the next corollary generalizing the fact that a self-expander without boundary in $\rr^m$ cannot be compact \cite[Prop.~5.3]{cao-li}.

\begin{corollary}
\label{cor:non-compact}
There are no compact weighted minimal $n$-dimensional submanifolds in $M^m_w$ with respect to a radial weight $e^h$ such that $H_n^h(t)>0$ for any $t>0$. 
\end{corollary} 

By combining Corollary~\ref{cor:radial2} with previous computations, we can give some consequences for $h$-minimal submanifolds in $M^m_w$ with weight $w(r)^k$.

\begin{proposition}
\label{prop:ball3}
Consider a model space $M^m_w$ where the function $H$ in \eqref{eq:mcspheres} satisfies $H(t)\geq 0$ for any $t>0$. Let $P^n$ be a submanifold properly immersed in $M_w$ which is weighted minimal with respect to the weight $w(r)^k$ in $M_w\setminus\{0\}$. Suppose that:
\begin{itemize}
\item[(i)] $k<-n$ and $P\sub M_w\setminus B_{t_0}$, for some $t_0>0$, or
\item[(ii)] $k>-n$ and $P$ is compact with $o\notin P$.
\end{itemize}
Then, $P\subeq S_t$ for some $t>0$ with $H(t)=0$. Moreover, if $P$ is a hypersurface, then $P=S_t$. Hence, if $H>0$, then there are no weighted minimal submanifolds in any of the previous cases.
\end{proposition}

\begin{proof}
By Corollary~\ref{cor:radial2} a submanifold $P$ in the conditions of the statement is weighted parabolic. As in the proof of Theorem~\ref{th:ball1}, the function $v:=\psi(r)$ where $\psi(t):=\int_0^tw(s)\,ds$ satisfies
\[
\Delta^h_P\,v=H^h_n(r)\,w(r)=(n+k)\,H(r)\,w(r).
\]
Our hypotheses imply that, either $\Delta^h_P\,v\leq 0$ and $v$ is bounded from below, or $\Delta^h_P\,v\geq 0$ and $v$ is bounded from above. We conclude that $v$ is constant on $P$ and so, $P\subeq S_t$ for some $t>0$. Moreover, we have $P=S_t$ when $P$ is a hypersurface. Finally, since $\Delta^h_P\,v=0$ then $H(t)=0$.
\end{proof}

\begin{remark}
At first, we cannot exclude the existence of compact $h$-minimal submanifolds meeting the pole. In \cite[Ex.~4.4]{homostable} it is shown that round spheres passing through the origin are weighted minimal hypersurfaces in $\rr^m$ for the weight $r^{2-2m}$. 
\end{remark}

In the extremal case $k=-n$ we have $\Delta^h_P\,v=H^h_n(r)=0$, and we can deduce a corollary which, in the particular case of the radial homogeneous weight $r^{1-m}$ in $\rr^m$, improves a previous result for hypersurfaces of Ca\~nete and the third author \cite[Thm.~6.4]{homostable}.

\begin{corollary}
\label{cor:ball5}
Consider the weight $w(r)^{-n}$ in a model space $M^m_w$ such that $H(t)\geq 0$ for any $t>0$. Then, a weighted minimal submanifold $P$ properly immersed in $M_w$ such that, either $P\sub M_w\setminus B_{t_0}$ for some $t_0>0$, or $P$ is compact with $o\notin P$, must be contained in a metric sphere $S_t$. Moreover, if $n=m-1$ then $P=S_t$.
\end{corollary}

\subsection{Cylinder results}
\label{subsec:cylinder}
\noindent

In this section we follow an analysis similar to that in Section~\ref{subsec:ball} to study submanifolds of bounded weighted mean curvature inside or outside right solid cylinders in Euclidean space. We are interested in weights for which these cylinders have constant weighted mean curvature.

For $m\geq 3$ and $2\leq k\leq m-1$ we identify $\rr^m$ with $\rr^k\times\rr^{m-k}$. For any $t>0$ we consider the cylinder $\mathcal{C}_t:=S_t\times\rr^{m-k}$, where $S_t$ denotes the sphere about the origin in $\rr^k$ of radius $t$. The corresponding solid cylinder is $B_t\times\rr^{m-k}$, where $B_t$ is the open ball in $\rr^k$ bounded by $S_t$. Along the hypersurface $\cil_t$ we choose the unit normal $N(x,y):=(-x/t,0)$, where $(x,y)$ are the components of a point in $\rr^k\times\rr^{m-k}$. The Euclidean mean curvature of $\cil_t$ is
\[
H_{c}(t)=\frac{k-1}{(m-1)\,t}.
\]
According to \eqref{eq:fmc}, for a given weight $e^h$ in $\rr^m$, the $h$-mean curvature of $\cil_t$ is
\begin{equation}
\label{eq:hmccil2}
H^h_{c}(t)=\frac{k-1}{t}+\frac{1}{t}\,\escpr{\nabla h,X},
\end{equation}
where $X(x,y):=(x,0)$. 

For a radial weight $e^h$ with $h=f(r)$ we obtain
\[
H^h_{c}(t)=\frac{k-1}{t}+\frac{f'(r)}{r}\,t,
\]
which is constant on $\cil_t$ if and only if there are $b,c\in\rr$ such that $f(s)=b\,s^2+c$ for any $s\geq t$. For instance, in the Gaussian case $f(s)=-s^2/2$, we have $H^h_{c}(t)=(k-1)/t-t$. Thus, the unique $h$-minimal cylinder in this example is $\cil_{\sqrt{k-1}}$.

Now, we take a Gaussian perturbation $e^h$ with $h(x,y):=-r^2/2+\xi(d)+\rho(y)$, where $\xi$ and $\rho$ are $C^1$ functions with $\xi'(0)=0$, and $d$ is the horizontal norm $d(x,y):=|x|$. Then, we get
\begin{equation}
\label{eq:hmccil}
H^h_{c}(t)=\frac{k-1}{t}-t+\xi'(t),
\end{equation}
so that $\cil_t$ has constant $h$-mean curvature. Clearly $H^h_{c}(t)\to\infty$ when $t\to 0$. Moreover, if $\xi'$ is bounded from above, then $H^h_{c}(t)\to-\infty$ when $t\to\infty$. In this situation, for any $\la_0\in\rr$, there is $t_0>0$ such that $H^h_{c}(t_0)=\la_0$. Note that $H^h_{c}(t)$ is decreasing when $\xi$ is $C^2$ and concave, so that $t_0$ would be unique in that case. In particular, there would be a unique $h$-minimal cylinder $\cil_t$. 

Now, we can prove the main result of this section. The function $H^h_{c,n}(t)$ in the statement is defined as in \eqref{eq:hmccil} by replacing $k-1$ with $n$.

\begin{theorem}
\label{th:cylinder}
In $\rr^m=\rr^k\times\rr^{m-k}$ we consider a weight $e^h$, where $h(x,y):=-r^2/2+\xi(d)+\rho(y)$ for some $C^1$ functions $\xi$ and $\rho$ with $\xi'(0)=0$. Suppose that there are constants $\delta_1,\delta_2\geq 0$ such that $\xi'\leq\delta_1$ and $\escpr{(\nabla\rho)(y),y}\leq\delta_2\,|y|$ for any $y\in\rr^{m-k}$. 
\begin{itemize}
\item[(i)] For a fixed $\la_0\geq 0$, let $t_0>0$ be the first number such that  $H^h_{c,n}(t)\leq-\la_0$ $($resp. $H^h_c(t)\leq-\la_0$$)$ for any $t\geq t_0$. If $P^n$ is a submanifold properly immersed in $\rr^m$ with $|\ovh^h|\leq\la_0$ and $P\sub\rr^m\setminus\big(B_{t_0}\times\rr^{m-k}\big)$, then $P\subeq\cil_{t}$ for some $t\geq t_0$ such that $H^h_{c,n}(t)=-\la_0$ $($resp. $H^h_{c}(t)=-\la_0$$)$, and $\ovh^h=(\la_0/t)\,X$ on $P$. Moreover, if $P$ is a hypersurface, then $P=\cil_t$.
\item[(ii)] For a fixed $\la_0\geq 0$, let $t_0>0$ be the last number such that $H^h_{c}(t)\geq \la_0$ for any $t\leq t_0$. If $P$ is a hypersurface properly immersed in $\rr^m$ with $|\ovh^h|\leq\la_0$ and $P\sub\overline{B}_{t_0}\times\rr^{m-k}$, then $P=\cil_{t}$ for some $t\leq t_0$ where $H^h_{c}(t)=\la_0$.
\end{itemize}
\end{theorem} 

\begin{proof}
First we see that, for a weight $e^h$ as in the statement, any submanifold $P^n$ properly immersed in $\rr^m$ with bounded $h$-mean curvature is $h$-parabolic. For this we will apply a previous parabolicity criteria for weights $e^h$ where $h=f(r)+g$. If we define $g(x,y):=\xi(d)+\rho(y)$ then $g\in C^1(\rr^m)$ and $(\nabla g)(x,y)=\big(\xi'(d)\,x/d,(\nabla\rho)(y)\big)$. As a consequence, we have the equality
\[
\escpr{\nabla g,\nabla r}=\frac{1}{r}\,\left\{\xi'(d)\,d+\escpr{(\nabla\rho)(y),y}\right\},\quad\text{in } \rr^m\setminus\{0\}.
\]
It is clear that $\escpr{\nabla g,\nabla r}\leq\delta_1+\delta_2$ on $\rr^m\setminus\{0\}$. This allows us to invoke Corollary~\ref{cor:parabolicity}, which entails the $h$-parabolicity of $P$.

Now, we define the smooth function $v:\rr^m\to\rr$ by $v(x,y):=|x|^2/2=d(x,y)^2/2$. We take any submanifold $P^n\sub\rr^m$ and compute $\Delta^h_P\,v$. Note that $\nabla v=X$, and so $\nabla_P\,v=X^\top$, where $X(x,y):=(x,0)$. Thus, we get
\[
\Delta_P\,v=\divv_P(X-X^\bot)=\divv_PX+\escpr{n\ovh,X},
\]
where $\overline{H}_P$ is the mean curvature vector and we have used \eqref{eq:divnorm}. According to \eqref{weighted-Laplacian} and \eqref{eq:mcvector}, it follows that
\[
\Delta^h_P\,v=\divv_PX+\escpr{n\ovh,X}+\escpr{\nabla h,\nabla_P\,v}=\divv_PX+\escpr{\nabla h,X}+\escpr{\ovh^h,X}.
\]
On the other hand, observe that
\[
\divv_PX=\sum_{i=1}^n\escpr{D_{e_i}X,e_i}=\sum_{i=1}^n|e_i^\ell|^2,
\]
where $\{e_1,\ldots,e_n\}$ is an orthonormal basis of tangent vectors to $P$, and $z^{\ell}$ stands for the projection of a vector $z\in\rr^m$ onto $\rr^k\times\{0\}$. Let $\{N_1,\ldots,N_{m-n}\}$ be an orthonormal basis of vectors normal to $P$. For any coordinate vector field $\ptl_j$ in $\rr^m$ it is clear that
\[
1=|\ptl_j|^2=\sum_{i=1}^n\escpr{\ptl_j,e_i}^2+\sum_{i=1}^{m-n}\escpr{\ptl_j,N_i}^2.
\]
By summing up from $j=1$ to $j=k$, we deduce
\[
k=\sum_{i=1}^n|e_i^\ell|^2+\sum_{i=1}^{m-n}|N_i^\ell|^2.
\]
Thus, we have obtained
\begin{equation}
\label{eq:artemis}
\Delta^h_P\,v=\sum_{i=1}^n|e_i^\ell|^2+\escpr{\nabla h,X}+\escpr{\ovh^h,X}=k-\sum_{i=1}^{m-n}|N_i^\ell|^2+\escpr{\nabla h,X}+\escpr{\ovh^h,X}.
\end{equation}
At this point, we use that $(\nabla h)(x,y)=-(x,y)+\big(\xi'(d)\,x/d,(\nabla\rho)(y)\big)$ to conclude that
\begin{align}
\label{eq:lapcil}
\Delta^h_P\,v&=\sum_{i=1}^n|e_i^\ell|^2-d^2+\xi'(d)\,d+\escpr{\ovh^h,X}
\\
\nonumber
&=k-\sum_{i=1}^{m-n}|N_i^\ell|^2-d^2+\xi'(d)\,d+\escpr{\ovh^h,X}.
\end{align}

Now, suppose that we have the hypotheses in (i) with $H^h_{c,n}(t)\leq-\la_0$ for any $t\geq t_0$. The fact that $P$ is outside the solid cylinder $B_{t_0}\times\rr^{m-k}$ ensures that $v\geq t_0^2/2>0$ on $P$. The first equality in \eqref{eq:lapcil} together with the definition of $H^h_{c,n}$ and the Cauchy-Schwarz inequality yield the estimate
\[
\Delta^h_P\,v\leq n-d^2+\xi'(d)\,d+\escpr{\ovh^h,X}\leq
d\,\left(H^h_{c,n}(d)+\la_0\right)\leq 0.
\]
On the other hand, if the hypotheses in (ii) are satisfied, then $v\leq t_0^2/2$ on $P$. Moreover, the second equality in \eqref{eq:lapcil} and the same arguments as above imply that inequality
\[
\Delta^h_P\,v\geq k-1-d^2+\xi'(d)\,d+\escpr{\ovh^h,X}\geq
d\,\left(H^h_{c}(d)-\la_0\right)\geq 0
\]
holds on $P-\big(\{0\}\times\rr^{m-k}\big)$, which is a dense subset of $P$. In any case, the $h$-parabolicity of $P$ ensures that $v$ is constant on $P$, so that $P\subeq\cil_t$ for some $t\geq t_0$ or $t\leq t_0$. Indeed, when $P$ is a hypersurface then $P=\cil_t$ since $P$ is properly immersed and $\cil_t$ is connected. Note also that, since $\Delta^h_P\,v=0$, all the inequalities above must be equalities. This gives us the conclusions in the claim.

It remains to prove statement (i) with the hypothesis $H^h_c(t)\leq-\la_0$ for any $t\geq t_0$. We take the function $d=\sqrt{2v}$, which is smooth outside $\{0\}\times\rr^{m-k}$. By using the chain rule, we have
\[
\Delta^h_P\,d=\frac{\sqrt{2}}{2\,\sqrt{v}}\,\Delta^h_P\,v-\frac{\sqrt{2}}{4\,v^{3/2}}\,|\nabla_P v|^2=\frac{1}{d}\,\Delta^h_P\,v-\frac{1}{d^3}\,|\nabla_P v|^2.
\]
Note that
\[
|\nabla_P v|^2=|X-X^\bot|^2=|X|^2-\sum_{i=1}^{m-n}\escpr{X,N_i}^2=d^2\left(1-\frac{1}{d^2}\,\sum_{i=1}^{m-n}\escpr{X,N_i^\ell}^2\right).
\]
From the second equality in \eqref{eq:lapcil} and the inequality $(1/d^2)\,\sum_{i=1}^{m-n}\escpr{X,N_i^\ell}^2\leq\sum_{i=1}^{m-n}|N_i^\ell|^2$, we deduce
\begin{align*}
\Delta^h_P\,d\leq\frac{1}{d}\,\big(k-1-d^2+\xi'(d)\,d+\escpr{\ovh^h,X}\big)=H^h_c(d)+\frac{1}{d}\,\escpr{\ovh^h,X}\leq 0,
\end{align*}
and the proof finishes as in the previous cases.
\end{proof}

\begin{remark}
In general the hypothesis $H^h_{c,n}(t)\leq-\la_0$ for any $t\geq t_0$ is independent from the hypothesis $H^h_c(t)\leq-\la_0$ for any $t\geq t_0$.
\end{remark}

\begin{example}
Suppose that the function $\xi$ in the statement of Theorem~\ref{th:cylinder} is $C^2$ and concave. Then, there is a unique $t_0>0$ such that the cylinder $\cil_{t_0}$ is $h$-minimal. In this situation, it follows that an $h$-minimal hypersurface $P$ properly immersed in $\rr^m$ and contained in the interior or the exterior of $\cil_{t_0}$ coincides with $\cil_{t_0}$. In the particular case of the Gaussian weight (for which the critical cylinder is $\cil_{\sqrt{k-1}}$) this result was previously proved by Cavalcante and Espinar~\cite[Thms.~1.2 and 1.3]{espinar-halfspace} by a different method, see also \cite[Thm.~2]{pigola-rimoldi} and the recent paper of Impera, Pigola and Rimoldi~\cite[Thm.~A]{ipr}.
\end{example}

Another interesting case where the cylinders $\cil_t$ have constant weighted mean curvature occurs in $\rr^m=\rr^k\times\rr^{m-k}$ with weight $e^h$ such that $h(x,y):=\mu(y)$, for some function $\mu\in C^1(\rr^{m-k})$. For this weight it is clear from \eqref{eq:hmccil2} that $H^h_c(t)=(k-1)/t$ and so, none of the cylinders $\cil_t$ is $h$-minimal. In this situation we get this result.

\begin{proposition}
\label{prop:cylmin}
In $\rr^m=\rr^k\times\rr^{m-k}$ we consider a weight $e^h$, where $h(x,y):=\mu(y)$ for some function $\mu\in C^1(\rr^{m-k})$. Let $P^n$ be an $h$-minimal submanifold with $n\geq m-k$ immersed in some solid cylinder $\overline{B}_{t_0}\times\rr^{m-k}$.
\begin{itemize}
\item[(i)] If $n>m-k$ then $P$ is $h$-hyperbolic.
\item[(ii)] If $n=m-k$ and $P$ is a complete $h$-parabolic submanifold, then $P=\{x_0\}\times\rr^{m-k}$ for some $x_0\in\rr^k$, and $\rr^{m-k}$ is weighted parabolic for the weight $e^\mu$.
\end{itemize}
In particular, there are no compact $h$-minimal submanifolds of dimension $n\geq m-k$.
\end{proposition}

\begin{proof}
From the second equality in \eqref{eq:artemis}, the function $v(x,y):=|x|^2/2$ satisfies
\[
\Delta^h_P\,v=k-\sum_{i=1}^{m-n}|N_i^\ell|^2+\escpr{\nabla h,X}+\escpr{\ovh^h,X}\geq k-(m-n)\geq 0
\]
whereas $v\leq t_0^2/2$. The previous estimate on $\Delta^h_P\,v$ shows that $v$ is not a constant function when $n>m-k$, so that $P$ is $h$-hyperbolic in this case. On the other hand, if $n=m-k$ and $P$ is $h$-parabolic, then $v$ is constant on $P$. This implies that $N_i^\ell=N_i$ for any $i=1,\ldots,k$ and so, $T_pP=\{0\}\times\rr^{m-k}$ for any $p\in P$. From this we deduce that the horizontal projection $(x,y)\mapsto x$ is constant on $P$, i.e., there is $x_0\in\rr^k$ such that $P\subeq\{x_0\}\times\rr^{m-k}$. Moreover, we get the equality provided $P$ is complete. Finally, the $h$-parabolicity of $\{x_0\}\times\rr^{m-k}$ is equivalent to the $\mu$-parabolicity of $\rr^{m-k}$.
\end{proof}

\begin{example}
The result applies in $\rr^m=\rr^{m-1}\times\rr$ with weight $e^h$ such that $h(x,t):=t$. This is a relevant situation since the $h$-minimal hypersurfaces are the translating solitons of the mean curvature flow. In relation to this P\'erez-Garc\'ia~\cite[Thm.~2.2]{perez-garcia} has shown non-existence of non-compact embedded $(m-1)$-dimensional translators contained in any cylinder of $\rr^m$.
\end{example}

\subsection{Half-space results}
\label{subsec:half-space}
\noindent

In this section we analyze the height function with respect to a unit vector in $\rr^m$ to characterize submanifolds of bounded $h$-mean curvature within a closed half-space. We begin by computing the mean curvature of Euclidean hyperplanes with respect to some weights.

Take a weight $e^h$ in $\rr^m$ where $h:=f(r)+g$, for some $C^1$ functions $f(r),g:\rr^m\to\rr$. Consider the hyperplane $\mathcal{L}_t:=\{p\in\rr^m\,;\,\escpr{p,a}=t\}$, where $a\in\rr^m$ is a unit vector and $t\in\rr$. The closed half-spaces determined by $\ele_t$ are the sets $\ele_t^+:=\{p\in\rr^m\,;\,\escpr{p,a}\geq t\}$ and $\ele_t^-:=\{p\in\rr^m\,;\,\escpr{p,a}\leq t\}$. From \eqref{eq:fmc}, the $h$-mean curvature of $\ele_t$ is given by
\begin{equation}
\label{eq:fmchyp}
H^h_{\ele_t}=-\escpr{\nabla h,a}=-\escpr{f'(r)\,\nabla r+\nabla g,a}=-\frac{f'(r)}{r}\,t-\frac{\ptl g}{\ptl a},
\end{equation}
where the last equality holds on $\ele_t\setminus\{0\}$. So, the linear hyperplane $\ele_0$ is $h$-minimal if and only if $\ptl g/\ptl a=0$ on $\ele_0$. In particular, for a radial weight $h=f(r)$ all the linear hyperplanes in $\rr^m$ are $h$-minimal. Indeed, these are the unique $h$-minimal hyperplanes unless $f$ is constant for $s\geq s_0$. 

Now, we establish a half-space result for $h$-minimal submanifolds inside a linear half-space. 

\begin{theorem}
\label{th:half-space1}
In $\rr^m$ we consider a weight $e^h$ where $h:=f(r)+g$, for some $f(r),g\in C^1(\rr^m)$ such that $f$ is non-increasing and $f'(s)\to-\infty$ when $s\to\infty$. Let $a\in\rr^m$ be a unit vector for which $\ptl g/\ptl a=0$ on the associated linear hyperplane $\ele_0$. Suppose also that:
\begin{itemize}
\item[(i)] $\ptl g/\ptl a\geq 0$ on $\ele_0^-$ $($resp. $\ptl g/\ptl a\leq 0$ on $\ele_0^+$$)$,
\item[(ii)] in $\ele_0^-\setminus\{0\}$ $($resp. $\ele_0^+\setminus\{0\}$$)$ we have $\escpr{\nabla g,\nabla r}\leq\delta$ for some constant $\delta\in\rr$.
\end{itemize} 
If $P^n$ is an $h$-minimal submanifold properly immersed in $\rr^m$ with $P\sub\ele_0^-$ $($resp. $P\sub\ele_0^+$$)$, then $P\subeq\ele_t$ for some $t\leq 0$ $($resp. $t\geq 0$$)$ and $\ptl g/\ptl a=0$ on $P$. In case $t\neq 0$ then $f'(r)=0$ on $P$. As a consequence $P=\ele_0$ when $P$ is a hypersurface.
\end{theorem}

\begin{proof}
A submanifold $P^n$ in the conditions of the statement is $h$-parabolic. This is clear if $P$ is compact; in the non-compact case we deduce the $h$-parabolicity  from Corollary~\ref{cor:parabolicity}.  

Consider the height function $\pi:\rr^m\to\rr$ defined by $\pi(p):=\escpr{p,a}$. Let us compute $\Delta^h_P\pi$. Since $\nabla\pi=a$ then $\nabla_P\pi=a^\top$, and so
\[
\Delta_P\pi=\divv_P(a-a^\bot)=\escpr{n\overline{H}_P,a},
\]
where $\overline{H}_P$ is the mean curvature vector and we have used \eqref{eq:divnorm}. From \eqref{weighted-Laplacian} and \eqref{eq:mcvector} we obtain
\begin{align}
\nonumber
\Delta^h_P\pi&=\Delta_P\pi+\escpr{\nabla h,\nabla_P\pi}=\escpr{n\overline{H}_P,a}+\escpr{\nabla h,a}-\escpr{\nabla h,a^\bot}
\\
\label{eq:laplace-height}
&=\escpr{\ovh^h,a}+\escpr{\nabla h,a}.
\end{align}
By taking into account that $P$ is $h$-minimal and $h=f(r)+g$, we get
\[
\Delta^h_P\pi=\escpr{\nabla h,a}=\frac{f'(r)}{r}\,\pi+\frac{\ptl g}{\ptl a}, \quad\text{on} \ P\setminus\{0\}.
\]
Now, our hypotheses imply that $\Delta^h_P\pi\geq 0$ and $\pi\leq 0$ on $P\setminus\{0\}$ (resp. $\Delta^h_P\pi\leq 0$ and $\pi\geq 0$ on $P\setminus\{0\}$). By continuity these inequalities are valid on the whole submanifold $P$. Thus the $h$-parabolicity of $P$ yields that $\pi$ is constant on $P$, so that $P\subeq\ele_t$ for some $t\leq 0$ (resp. $t\geq 0$). From the fact that $\Delta^h_P\pi=0$ it follows that $f'(r)\,t=0$ and $\ptl g/\ptl a=0$ on $P$. Thus, it is clear that $f'(r)=0$ on $P$ provided $t\neq 0$. Finally, if $n=m-1$, then $P=\ele_t$ because $P$ is properly immersed in $\rr^m$ and $\ele_t$ is connected. Moreover, since $f'(r)\,t=0$ on $P$ and $f'(s)\to -\infty$ when $s\to\infty$ we infer that $t=0$. Thus $P=\ele_0$ and the proof is complete.
\end{proof}

\begin{examples}
1. The result applies when the perturbation term $g$ does not depend on the coordinate $x_i$ and $\escpr{\nabla g,\nabla r}\leq\delta$ in $\rr^m\setminus\{0\}$. In this case we deduce that any submanifold $P^n$ properly immersed in $\rr^m$ and contained in one of the half-spaces $x_i\leq 0$ or $x_i\geq 0$ is inside a hyperplane $x_i=t$, with $t=0$ when $n=m-1$.

2. Consider a radial weight $e^h$ where $h=f(r)$ and $f$ is decreasing with $f'(s)\to-\infty$ when $s\to\infty$. In this situation we know from \eqref{eq:fmchyp} that the unique $h$-minimal hyperplanes are the linear ones. By Theorem~\ref{th:half-space1} any $h$-minimal submanifold $P$ properly immersed in $\rr^m$ and contained within a closed linear half-space $\ele_0^+$ or $\ele_0^-$ satisfies $P\subeq\ele_0$, with equality $P=\ele_0$ provided $P$ is a hypersurface. This holds when $f(s):=as^k$ with $a<0$ and $k>1$. For $k=2$ we recover the half-space theorem for self-shrinker hypersurfaces proved by Pigola and Rimoldi~\cite[Thm.~3]{pigola-rimoldi}, see also
Cavalcante and Espinar~\cite[Thm.~1.1]{espinar-halfspace} for a different proof. The result is still valid for Gaussian perturbations of the form $h(r):=-r^2/2+g(r)$, where $g$ is a $C^2$ concave function with $g'(0)=0$.

3. Sometimes we can apply Theorem~\ref{th:half-space1} with a perturbation term $g$ which is not radial and involves all the coordinates $x_i$. For instance, if $g(x,t):=\xi(x)+\rho(t)$, where $\xi\in C^1(\rr^{m-1})$ with $\escpr{(\nabla\xi)(x),x}\leq\delta\,|x|$ for some $\delta\geq 0$, and $\rho\in C^1(\rr)$ with $t\,\rho'(t)\leq 0$ for any $t\in\rr$, then the unique proper $h$-minimal hypersurface contained inside $t\leq 0$ or $t\geq 0$ is the hyperplane $t=0$. The same conclusion holds for $g(x,t):=-t^2\,|x|^2/2$.
\end{examples}

Our next aim is to prove a half-space theorem for submanifolds of bounded mean curvature (possibly non-minimal) with respect to a weight $e^h$ in $\rr^m$ with $h=f(r)+g$. We first analyze when a hyperplane $\ele_t$ with $t\neq 0$ has constant $h$-mean curvature. In the radial case $h=f(r)$ equation~\eqref{eq:fmchyp} implies that $H^h_{\ele_t}$ is constant if and only if there are $b,c\in\rr$ such that $f(s)=b\,s^2+c$ for any $s\geq|t|$. This is the case of the Gaussian weight $e^{-r^2/2}$, where $H^h_{\ele_t}=t$. On the other hand, for a Gaussian perturbation $e^h$ with $h=-r^2/2+g$, we have $H^h_{\ele_t}=t-\la$ if and only if $\ptl g/\ptl a=\la$ on $\ele_t$. For these weights we can establish an extension of Theorem~\ref{th:half-space1} which is valid for submanifolds within a closed half-space and having bounded $h$-mean curvature. In the particular case of $\lambda$-hypersurfaces (those with constant weighted mean curvature for the Gaussian weight) we deduce a result obtained by Cavalcante and Espinar~\cite[Thm.~1.4]{espinar-halfspace} from a different method.

\begin{theorem}
\label{th:half-space2}
In $\rr^m$ we consider a weight $e^h$ where $h:=-r^2/2+g$, for some $g\in C^1(\rr^m)$. Let $a\in\rr^m$ be a unit vector such that $\ptl g/\ptl a=\la$ on the associated hyperplane $\ele_{t_0}$. Suppose also that:
\begin{itemize}
\item[(i)] $\ptl g/\ptl a\geq\la$ on $\ele_{t_0}^-$ $($resp. $\ptl g/\ptl a\leq\la$ on $\ele_{t_0}^+$$)$,
\item[(ii)] in $\ele_{t_0}^-\setminus\{0\}$ $($resp. $\ele_{t_0}^+\setminus\{0\}$$)$ we have $\escpr{\nabla g,\nabla r}\leq\delta$ for some constant $\delta\in\rr$.
\end{itemize} 
If a submanifold $P^n$ properly immersed in $\rr^m$ satisfies $|\ovh^h|\leq \la-t_0$ $($resp. $|\ovh^h|\leq t_0-\la$$)$ and $P\sub\ele_{t_0}^-$ $($resp. $P\sub\ele_{t_0}^+$$)$, then $P\subeq\ele_{t_0}$ and $\ovh^h=(t_0-\la)\,a$ on $P$. Moreover, if $P$ is a hypersurface then $P=\ele_{t_0}$.
\end{theorem}

\begin{proof}
The proof is similar to that of Theorem~\ref{th:half-space1}. First, we apply Corollary~\ref{cor:parabolicity} to deduce that $P$ is $h$-parabolic. By equation \eqref{eq:laplace-height}, the function $\pi:\rr^m\to\rr$ defined by $\pi(p):=\escpr{p,a}$ satisfies
\[
\Delta^h_P\pi=\escpr{\ovh^h,a}+\escpr{\nabla h,a}=\escpr{\ovh^h,a}-\pi+\frac{\ptl g}{\ptl a}.
\]
Our hypotheses and the Cauchy-Schwarz inequality yield that $\Delta^h_P\pi\geq 0$ and $\pi\leq t_0$ on $P$ (resp. $\Delta^h_P\pi\leq 0$ and $\pi\geq t_0$ on $P$). From the $h$-parabolicity of $P$ we conclude that $\pi$ is constant on $P$. Now, the identity $\Delta^h_P\pi=0$ on $P$ gives us $P\subeq\ele_{t_0}$ and $\ovh^h=(t_0-\la)\,a$ on $P$. In case $n=m-1$ we get $P=\ele_{t_0}$. This finishes the proof.
\end{proof}

Next, we analyze submanifolds within a closed half-space and having bounded mean curvature with respect to a weight $e^h$ in $\rr^m=\rr^{m-1}\times\rr$, where $m\geq 2$ and $h(x,t):=\mu(t)$ for some $\mu\in C^1(\rr)$. This context includes the translating solitons of the mean curvature flow, which appears as $h$-minimal submanifolds in the special case where $\mu(t)=t$. 
In the previous situation, by a \emph{horizontal hyperplane} we mean a hyperplane $\ele_t$ with unit normal $a=\ptl_t:=(0,1)$. A \emph{vertical hyperplane} is one for which $\escpr{a,\ptl_t}=0$. From \eqref{eq:fmchyp} we see that $H^h_{\ele_t}=-\mu'(t)$ provided $\ele_t$ is horizontal, whereas any vertical hyperplane is $h$-minimal. A hyperplane which is neither horizontal nor vertical has constant $h$-mean curvature if and only if $\mu$ is an affine function.

Now, we can prove this statement.

\begin{theorem}
\label{th:half-space3}
In $\rr^m=\rr^{m-1}\times\rr$ we consider a weight $e^h$ with $h(x,t):=\mu(t)$ for some $\mu\in C^1(\rr)$. Let $P^n$ be an $h$-parabolic submanifold in $\rr^m$.
\begin{itemize}
\item[(i)] Suppose that there is a constant $c\geq 0$ such that $\mu'\geq c$ $($resp. $\mu'\leq -c$$)$ on $\rr$. If $|\ovh^h|\leq c$ and $P\sub\ele^-_{t_0}$ $($resp. $P\sub\ele^+_{t_0}$$)$ for some horizontal hyperplane $\ele_{t_0}$, then $P$ is contained in some horizontal hyperplane $\ele_t$ with $\mu'(t)=c$ $($resp. $\mu'(t)=-c$$)$ and $\ovh^h=-c\,\ptl_t$ $($resp. $\ovh^h=c\,\ptl_t$$)$ on $P$. Moreover, if $P$ is a complete hypersurface, then $P=\ele_t$ and $m\in\{2,3\}$.
\item[(ii)] Suppose that $\mu'\geq 0$ $($resp. $\mu'\leq 0$$)$ on $\rr$. Consider a unit vector $a\in\rr^m$ such that $\escpr{a,\ptl_t}\in(0,1)$. If $P$ is $h$-minimal and $P\sub\ele^-_{t_0}$ $($resp. $P\sub\ele^+_{t_0}$$)$ for some hyperplane $\ele_{t_0}$ with unit normal $a$, then $P\subeq\ele_t$ for some $t\in\rr$ and $\mu$ is constant over the vertical projection of $P$. In case $P$ is a complete hypersurface then $P=\ele_t$ and $\mu$ is constant.
\item[(iii)] If $P$ is $h$-minimal and contained inside a vertical half-space $\ele^-_{t_0}$ or $\ele_{t_0}^+$, then $P$ is contained in some vertical hyperplane $\ele_t$. Moreover, if $P$ is a complete hypersurface, then $P=\ele_t$.
\end{itemize}
\end{theorem}

\begin{proof}
Let $a\in\rr^m$ be a unit vector. From equation~\eqref{eq:laplace-height} the height function $\pi(p):=\escpr{p,a}$ satisfies
\[
\Delta^h_P\pi=\escpr{\ovh^h,a}+\escpr{\nabla h,a}=\escpr{\ovh^h,a}+\mu'(t)\,\escpr{\ptl_t,a}.
\]
To prove (i) we take $a=\ptl_t$. Our hypotheses and the Cauchy-Schwarz inequality imply that $\Delta^h_P\pi\geq 0$ and $\pi\leq t_0$ on $P$ (resp. $\Delta^h_P\pi\leq 0$ and $\pi\geq t_0$ on $P$). Since we assume that $P$ is $h$-parabolic then $\pi$ is constant on $P$, i.e., there is a horizontal hyperplane $\ele_t$ such that $P\subeq\ele_t$. From here the equality $\Delta^h_P\pi=0$ leads to the conclusions in the statement. In case $n=m-1$ and $P$ is complete, then $P=\ele_t$. As the weight $e^h$ is constant on $\ele_t$ then the $h$-parabolicity of $\ele_t$ is equivalent to the classical (Riemannian) parabolicity. Finally, the well-known fact that $\rr^{m-1}$ is parabolic if and only if $m\in\{2,3\}$ (which follows for instance from Corollary~\ref{cor:ahlfors}) provides the claim. The proofs of (ii) and (iii) are similar.
\end{proof}

The previous theorem entails the following hyperbolicity result.

\begin{corollary}
\label{cor:3.1}
In $\rr^m=\rr^{m-1}\times\rr$ we consider a weight $e^h$ with $h(x,t):=\mu(t)$ for some $\mu\in C^1(\rr)$. Suppose that $\mu'\geq c$ $($resp. $\mu'\leq -c$$)$ on $\rr$ for some constant $c\geq 0$. Let $P$ be a submanifold in $\rr^m$ with $|\ovh^h|\leq c$ and such that $P\sub\ele^-_{t_0}$ $($resp. $P\sub\ele^+_{t_0}$$)$ for some horizontal hyperplane $\ele_{t_0}$. If equality $|\mu'|=c$ never holds on $\rr$, or $|\ovh^h|<c$ on $P$, or $m\geq 4$ and $P$ is a complete hypersurface, then $P$ is $h$-hyperbolic.
\end{corollary}

In the particular case of $h$-minimal submanifolds we can gain more information from our hyperbolicity criterion in Theorem~\ref{hyperbolicity-sub-model}.

\begin{corollary}
\label{cor:3.2}
In $\rr^m=\rr^{m-1}\times\rr$ we consider a weight $e^h$ with $h(x,t):=\mu(t)$ for some $\mu\in C^1(\rr)$. Suppose that there is a constant $c\geq 0$ such that $\mu'\geq c$ $($resp. $\mu'\leq -c$$)$ on $\rr$. Let $P^n$ be an $h$-minimal submanifold properly immersed in $\rr^m$.
\begin{itemize}
\item[(i)] If $c>0$ and $P\sub\ele^-_{t_0}$ $($resp. $P\sub\ele^+_{t_0}$$)$ for some horizontal hyperplane $\ele_{t_0}$, then $P$ is $h$-hyperbolic.
\item[(ii)] If $c>0$ and $P$ is a hypersurface with $P\sub\ele^-_{t_0}$ $($resp. $P\sub\ele^+_{t_0}$$)$ for some non-vertical hyperplane $\ele_{t_0}$, then $P$ is $h$-hyperbolic. Moreover, if the hypersurface $P$ is contained inside a closed vertical half-space, then $P$ is also $h$-hyperbolic.
\item[(iii)] If $c=0$, $n\geq 3$, and $P$ is contained in the horizontal half-space $\ele^+_0$ $($resp. $\ele^-_0$$)$, then $P$ is $h$-hyperbolic.
\end{itemize}
\end{corollary}

\begin{proof}
Statement (i) is clear from Corollary~\ref{cor:3.1}. To prove the first part of statement (ii), we suppose that $\ele_{t_0}$ is a non-horizontal hyperplane and we reason by contradiction. By assuming that $P$ is $h$-parabolic we would get from Theorem~\ref{th:half-space3} (ii) that $\mu$ is constant, which contradicts that $\mu'\geq c$ or $\mu'\leq -c$ with $c>0$. If we suppose that $P$ is a hypersurface inside a closed vertical half-space, then we know from Theorem~\ref{th:half-space3} (iii) that $P$ is a vertical hyperplane $\ele_t$. Let us see that $\ele_t$ is $h$-hyperbolic. In the case $\mu'\geq c$ we define the function $v:=e^{-c\pi}$, where $\pi(p):=\escpr{p,\ptl_t}$. An easy computation gives $\Delta^h_P\,v=c\,\big(c-(\mu'\circ\pi)\big)\,v\leq 0$. Since $v$ is a non-constant positive function it follows that $\ele_t$ is $h$-hyperbolic. In the case $\mu'\leq -c$ we can reason in a similar way with the function $v:=e^{c\pi}$. Finally, we prove (iii) by using Theorem~\ref{hyperbolicity-sub-model}. Note that the submanifold $P$ cannot be compact by Proposition~\ref{prop:cylmin}. On the other hand, the $h$-minimality of $P$ together with our hypotheses yield
\[
\escpr{\nabla h,\nabla r}+\escpr{\ovh^h,\nabla r}=\frac{(\mu'\circ\pi)\,\pi}{r}\geq 0, \quad\text{on } P\setminus\{0\}.
\]
Thus the conditions (A) and (B) in the statement of Theorem~\ref{hyperbolicity-sub-model} hold. Moreover, the integrability condition \eqref{eq:icII} is also satisfied since $n\geq 3$. This completes the proof.
\end{proof}

In the special case $\mu(t)=t$, for which the associated $h$-minimal submanifolds are the translating solitons of the mean curvature flow, we can deduce an extension of Corollary~\ref{cor:3.2} (iii) which is a direct consequence of Theorem~\ref{hyperbolicity-sub-model}.

\begin{corollary}
\label{cor:translating}
In $\rr^m=\rr^{m-1}\times\rr$ we consider the weight $e^h$ with $h(x,t):=t$. Let $P^n$ be an $h$-minimal submanifold properly immersed in $\rr^m$, and such that
\[
P\sub\{(x,t)\in\rr^m\,;\,t\geq r\,\alpha(r)\},
\]
where the function $\alpha(t)$ satisfies:
\begin{itemize}
\item[(i)] $\alpha(t)\geq -n/t$, for any $t\geq t_0$,
\item[(ii)] $\int_{t_0}^{\infty}t^{1-n}\exp(-\int_{t_0}^t\alpha(s)\,ds)\,dt<\infty$.
\end{itemize}
Then, $P$ is $h$-hyperbolic.
\end{corollary}

\begin{example}
This hyperbolicity criterion is valid when $P^n$ is contained in a horizontal half-space $\ele^+_{t_0}$ with $t_0>2-n$. In particular some examples of complete translating solitons, like the grim hyperplane, the translating paraboloid and the translating catenoid (see \cite[Sect.~2]{css} and \cite[Sect.~2.2]{martin} for more precise descriptions) are $h$-hyperbolic. The criterion also applies when $P^n\sub\{(x,t)\in\rr^m\,;\,t\geq r^{k+1}\}$ provided $n\geq 3$ and $k>-1$.
\end{example}

We finish this section with some comments about the existence of compact submanifolds of bounded $h$-mean curvature. 

\begin{remark}
Take a weight $e^h$ in $\rr^m=\rr^{m-1}\times\rr$ with $h(x,t):=\mu(t)$. From Theorem~\ref{th:half-space3} (i) we know that, if $\mu'\geq c$ or $\mu'\leq -c$ for some constant $c\geq 0$, then a hypersurface $P$ such that $|\ovh^h|\leq c$ cannot be compact. Indeed, by using Theorem~\ref{th:half-space3} (iii) or Proposition~\ref{prop:cylmin}, it follows that there are no compact $h$-minimal hypersurfaces (in this case no further assumption on $\mu'$ is needed). Recently, L\'opez~\cite[Thm.~4.1]{rafa} has shown non-existence of compact surfaces in $\rr^3$ with constant $h$-mean curvature for $h(x,t)=t$. His argument can be extended to $\rr^m$ and any function $h(x,t)=\mu(t)$ such that $\mu'$ never vanishes. Therefore, for such weights, if a solution to the isoperimetric problem of minimizing the weighted area in \eqref{eq:volarea} for fixed weighted volume exists, then it cannot be compact. 
\end{remark}

\subsection{A Bernstein-type result}
\label{subsec:bernstein}
\noindent

Here we characterize entire horizontal graphs in $\rr^m=\rr^{m-1}\times\rr$ having constant mean curvature with respect to suitable weights. It is well known that, in the unweighted context, the unique solutions when $m\leq 7$ are minimal hyperplanes, see \cite[Ch.~17]{giusti}. Motivated by this fact we seek weights for which the solutions are also hyperplanes. This leads us to consider a weight $e^h$ with $h(x,t):=\eta(x)+\mu(t)$, for some functions $\eta\in C^1(\rr^{m-1})$ and $\mu\in C^1(\rr)$. In this situation, equation \eqref{eq:fmchyp} implies that any horizontal hyperplane $\ele_t$ has constant $h$-mean curvature $H^h_{\ele_t}=-\mu'(t)$.
 
Our main result provides some restrictions on $\eta$ and $\mu$ ensuring that the unique solutions to the Bernstein problem in any dimension are the horizontal hyperplanes.

\begin{theorem}
\label{th:bernstein}
In $\rr^{m}=\rr^{m-1}\times\rr$ we take a weight $e^h$ of the form $h(x,t):=\eta(x)+\mu(t)$, for some functions $\eta\in C^2(\rr^{m-1})$ and $\mu\in C^2(\rr)$ such that:
\begin{itemize}
\item[(i)] there is a continuous function $\beta:\rr^+\to\rr$ such that $\escpr{\nabla h,\nabla r}\leq\beta(r)$ on $\rr^m\setminus\{0\}$ and $\beta(s)\to -\infty$ when $s\to\infty$,
\item[(ii)] there is a constant $c\in\rr$ such that $(\emph{Hess}\,\eta)_x(X,X)\leq c\leq\mu''(t)$, for any $x,X\in\rr^{m-1}$ with $|X|=1$, and any $t\in\rr$.
\end{itemize}
If $P$ is a graph $t=\var(x)$ where $\var\in C^3(\rr^{m-1})$, and $P$ has constant $h$-mean curvature, then $P$ is a hyperplane. Moreover, if some of the inequalities in \emph{(ii)} is always strict, then $P$ coincides with some horizontal hyperplane $\ele_t$. 
\end{theorem}

\begin{proof}
Let $P$ be a graph as in the statement. By continuity of $\var$ the hypersurface $P$ is a closed subset of $\rr^m$ homeomorphic to $\rr^{m-1}$. In particular, $P$ is a non-compact and properly embedded hypersurface in $\rr^{m}$. By using hypothesis (i) and the fact that $P$ has constant $h$-mean curvature, it follows from Corollary~\ref{cor:useful} that $P$ is $h$-parabolic. On the other hand, since $P$ is the graph $t=\var(x)$, we can define the unit normal
\[
N:=\frac{(\nabla\var,-1)}{\sqrt{1+|\nabla\var|^2}},
\]
for which the associated angle function $\theta:=\escpr{N,\ptl_t}$ satisfies $\theta<0$ on $P$. We denote by $n$ the projection of $N$ onto $\rr^{m-1}$. Hence we have $N=(n,\theta)$, so that $|n|^2+\theta^2=1$.

Note that $\theta\in C^2(P)$. Let us see that its weighted Laplacian is given by  
\begin{equation}
\label{eq:laplaceangle}
\Delta^h_{P}\,\theta=\left\{(\text{Hess}\,\eta)(n,n)+(\mu''\circ\pi)\,(\theta^2-1)-|\sg|^2\right\}\theta,
\end{equation}
where $\pi(x,t):=t$ is the vertical height function and $\sg$ is the Euclidean second fundamental form of $P$ with respect to $N$.

To prove \eqref{eq:laplaceangle} consider the group $\{\tau_s\}_{s\in\rr}$ of the vertical translations in $\rr^{m}$. Denote $P_s:=\tau_s(P)$ and define the unit normal $N_s$ on $P_s$ by $N_s(\tau_s(p)):=N(p)$, for any $p\in P$. By taking into account that $P$ has constant $h$-mean curvature we have the following identity 
\[
\frac{d}{ds}\bigg|_{s=0}H^h_{P_s}\circ\tau_s=\Delta^h_{P}\,\theta+\big(\text{Ric}_h(N,N)+|\sg|^2\big)\,\theta,
\]
see \cite[Eq.~(3.5)]{castro-rosales} and the references therein. Here $\ric_h$ is the Bakry-\'Emery-Ricci tensor defined by $\ric_h:=-\text{Hess}\,h$. Take $p=(x,t)\in P$. Having in mind \eqref{eq:fmc} and that any $\tau_s$ is an isometry, we get
\begin{align*}
H^h_{P_s}(\tau_s(p))&=\big((m-1)\,H_{P_s}-\escpr{\nabla h,N_s}\big)(\tau_s(p))=(m-1)\,H_P(p)-\escpr{(\nabla h)(x,t+s),N(p)}
\\
&=(m-1)\,H_P(p)-\escpr{\big((\nabla\eta)(x),\mu'(t+s)\big),N(p)}.
\end{align*}
As a consequence
\[
\big(\Delta^h_{P}\,\theta+(\text{Ric}_h(N,N)+|\sg|^2)\,\theta\big)(p)=-\escpr{\big(0,\mu''(t)\big),N(p)}=-(\mu''\circ\pi)(p)\,\theta(p).
\]
On the other hand, it is straightforward to check that
\[
\ric_h(N,N)=-(\text{Hess}\,\eta)(n,n)-(\mu''\circ\pi)\,\theta^2.
\]
By substituting this equality into the previous one, and simplifying, we obtain \eqref{eq:laplaceangle}.

At this point, hypothesis (ii) and the fact that $-1\leq\theta<0$ yield
\[
\Delta^h_P\,\theta\geq\big(c\,|n|^2+c\,\theta^2-c-|\sigma|^2\big)\,\theta=-|\sg|^2\,\theta\geq 0.
\]
Thus, since $P$ is $h$-parabolic, we deduce that $\theta$ is constant on $P$. Moreover, the equality $\Delta^h_P\,\theta=0$ gives us $(\text{Hess}\,\eta)(n,n)=c\,|n|^2$, $(\mu''\circ\pi)\,(\theta^2-1)=c\,(\theta^2-1)$ and $|\sg|^2=0$ on $P$. In particular $P$ is a hyperplane. Moreover, if one of the inequalities in (ii) is always strict then $n=0$ or $\theta=-1$ on $P$. In both cases $N=-\ptl_t$ on $P$, so that $P$ is a horizontal hyperplane $\ele_t$.
\end{proof}

\begin{example}
Consider a perturbation $e^h$ of the Gaussian weight with $h:=-r^2/2+g$. From \eqref{eq:fmchyp} it is not difficult to see that all the horizontal hyperplanes have constant $h$-mean curvature if and only if $g(x,t)=\xi(x)+\rho(t)$, for some $\xi:\rr^{m-1}\to\rr$ and $\rho:\rr\to\rr$. If we assume that:
\begin{itemize}
\item[(i)] $\xi$ is a concave $C^2$ function with $\escpr{(\nabla\xi)(x),x}\leq \delta_1|x|$, for some constant $\delta_1\geq 0$,
\item[(ii)] $\rho$ is a convex $C^2$ function with $t\,\rho'(t)\leq\delta_2\,|t|$, for some constant $\delta_2\geq 0$,
\end{itemize}
then we can apply Theorem~\ref{th:bernstein} to conclude that any entire $C^3$ horizontal graph $P$ with constant $h$-mean curvature is a hyperplane. Moreover, is $\xi$ is strictly concave or $\rho$ is not an affine function, then $P$ equals a horizontal hyperplane $\ele_t$. In general the solutions to the Bernstein problem need not be horizontal hyperplanes. Indeed, in the Gaussian setting $h=-r^2/2$ any non-vertical hyperplane is a solution.
\end{example}

\begin{remark}
The Bernstein problem for self-shrinker hypersurfaces was previously solved by Ecker and Huisken~\cite[App.]{ecker-huisken}, who assumed polynomial growth, and by Wang~\cite[Thm.~1.1]{wang} for arbitrary ones. For entire graphs of constant mean curvature in Gauss space the problem was treated by Cheng and Wei~\cite[Thm.~1.3]{cheng-wei}, who studied the Gauss map of properly immersed $\lambda$-hypersurfaces. Recently Doan~\cite{doan-bernstein2} has given another proof based on the fact that hyperplanes minimize the weighted area in Gauss space among hypersurfaces of the same weighted volume.
\end{remark}

\subsection{Strongly weighted stable hypersurfaces}
\label{subsec:stable}
\noindent

In the previous sections we have characterized hypersurfaces contained in certain regions (balls, cylinders and half-spaces) or having a special form (entire graphs). We finish this work by studying hypersurfaces that are \emph{stable} in a weighted sense that we precise below.

Consider a model space $M^m_w$ with a $C^2$ weight $e^h$, and a two-sided hypersurface $P$ immersed in $M_w$ with unit normal $N$. As we pointed out after equation \eqref{eq:fmc}, the hypersurface $P$ has constant $h$-mean curvature $H^h_P$ if and only if $P$ is a critical point of the weighted area $A_h$ in \eqref{eq:volarea} under variations for which the weighted volume $V_h$ is constant. This is also equivalent to that $(A_h-H^h_P\,V_h)'(0)=0$, for any compactly supported variation of $P$, see \cite[Prop.~3.2]{rcbm} and \cite[Cor.~3.3]{castro-rosales}. If, in addition, we have that $(A_h-H^h_P\,V_h)''(0)\geq 0$ for any compactly supported variation, then we say that $P$ is \emph{strongly $h$-stable}. From the second variation formulas in \cite[Prop.~3.6]{rcbm} and \cite[Prop.~3.5]{castro-rosales}, this is equivalent to the inequality
\[
Q_h(u,u)\geq 0,\quad\text{for any } u\in C^\infty_0(P).
\]
Here $Q_h$ is the \emph{$h$-index form} of $P$, which is defined by
\[
Q_h(u,u):=\int_P\left\{|\nabla_Pu|^2-\big(\ric_h(N,N)+|\sg|^2\big)\,u^2\right\}da_h,
\]
where $\ric_h:=\ric-\text{Hess}\,h$ is the \emph{Bakry-\'Emery-Ricci tensor}, and $|\sg|^2$ is the squared norm of the second fundamental form of $P$. 

In \cite[Thm.~3.1]{espinar}, Espinar employed gradient Schr\"odinger operators to show the following rigidity principle: if $\ric_h\geq 0$, then a complete $h$-parabolic and strongly $h$-stable hypersurface $P$ is totally geodesic and $\ric_h(N,N)=0$ on $P$. This result is also a consequence of Theorem~\ref{th:parchar} (iv). Indeed, if $\{\var_k\}_{k\in\nn}$ is a sequence as in the theorem, then the stability inequality above yields
\[
\int_P\big(\ric_h(N,N)+|\sg|^2\big)\,\var_k^2\,da_h\leq\int_P|\nabla_P\var_k|^2\,da_h, \quad\text{for any }k\in\nn.
\]
Hence, by passing to the limit and using Fatou's lemma, we get $\int_P\,(\ric_h(N,N)+|\sg|^2)\,da_h=0$, which proves the claim.

By using the rigidity principle Espinar showed in \cite[Sect.~4.1]{espinar} non-existence of complete $h$-parabolic and strongly $h$-stable self-shrinkers or translating solitons of the mean curvature flow. Indeed, this non-existence result is also valid for any weight $e^h$ in $\rr^m=\rr^{m-1}\times\rr$ where $h(x,t):=\mu(t)$ and $\mu$ is any $C^2$ concave function with $\mu'\geq c$ or $\mu'\leq- c$ for some $c>0$. To see this, note that such a weight satisfies $\ric_h(X,X)=-(\mu''\circ\pi)\,\escpr{\ptl_t,X}^2\geq 0$. Hence, a complete $h$-parabolic and strongly $h$-stable hypersurface must be a hyperplane. Since by \eqref{eq:fmchyp} the unique $h$-minimal hyperplanes are the vertical ones, and these are $h$-hyperbolic by Corollary~\ref{cor:3.2} (ii), the claim follows.

As a consequence of the rigidity principle, if we have a weight $e^h$ in $M^m_w$ such that $\ric_h\geq 0$ and all the hypersurfaces with constant $h$-mean curvature in a certain family are $h$-parabolic, then any strongly $h$-stable hypersurface in that family will be totally geodesic with $\ric_h(N,N)=0$. Thus, our parabolicity criteria in Section~\ref{sec:criteria} lead to characterization results for strongly $h$-stable hypersurfaces. For instance, from Corollary~\ref{cor:useful} we infer this fact.

\begin{theorem}
\label{cor:stable1}
Let $M^m_w$ be a model space such that $w\notin L^1(0,\infty)$ and the function $H$ in \eqref{eq:mcspheres} is bounded at infinity. Consider a $C^2$ weight $e^h$ such that $\emph{Ric}_h\geq 0$ and $\escpr{\nabla h,\nabla r}\leq\beta(r)$ on $M_w\setminus\{o\}$, for some continuous function $\beta:\rr^+\to\rr$ with $\beta(t)\to-\infty$ when $t\to\infty$. Then, any strongly $h$-stable hypersurface properly immersed in $M_w$ is totally geodesic and satisfies $\emph{Ric}_h(N,N)=0$.
\end{theorem}

The previous statement may be used to deduce non-existence of strongly $h$-stable hypersurfaces. In this direction we can derive the following consequence in Euclidean space.

\begin{corollary}
\label{cor:eustable}
In $\rr^m$ we consider a weight $e^h$ where $h:=f(r)+g$, for some functions $f(r),g\in C^2(\rr^m)$. If we suppose that
\begin{itemize}
\item[(i)] $f$ is concave with $f'(t)\to-\infty$ when $t\to\infty$,
\item[(ii)] $g$ is concave and there is a constant $\delta\in\rr$ such that $\escpr{\nabla g,\nabla r}\leq\delta$ on $\rr^m\setminus\{0\}$,
\end{itemize} 
then, there are no strongly $h$-stable hypersurfaces properly immersed in $\rr^m$. 
\end{corollary}

\begin{proof}
For a weight $e^h$ as in the statement it is clear that $\ric_h=-\text{Hess}\,f(r)-\text{Hess}\,g$. From the chain rule and equation \eqref{eq:hessr}, we get in $\rr^m\setminus\{0\}$ the identities
\begin{align*}
\big(\text{Hess}\,f(r)\big)(X,X)&=f''(r)\,\escpr{\nabla r,X}^2+f'(r)\,(\text{Hess}\,r)(X,X)
\\
&=f''(r)\,\escpr{\nabla r,X}^2+\frac{f'(r)}{r}\,\big(|X|^2-\escpr{\nabla r,X}^2\big).
\end{align*}
These computations show that $\ric_h\geq 0$ since $f$ and $g$ are concave. Thus, we can apply Theorem~\ref{cor:stable1} to deduce that, if a strongly $h$-stable hypersurface $P$ properly immersed in $\rr^m$ exists, then it must be a hyperplane where $\ric_h(N,N)=0$. In particular, we obtain that $\big(\text{Hess}\,f(r)\big)(N,N)=0$ on $P$, which is not possible by the last equation because $f'(t)\to-\infty$ when $t\to\infty$.
\end{proof}

\begin{example}
The corollary applies for weights $e^h$ with $h:=f(r)$, where $f$ is concave and $f'(t)\to-\infty$ when $t\to\infty$. This is the case of $f(t):=a\,t^k$ with $a<0$ and $k>1$, which includes the Gaussian weight. So, we generalize previous results of Colding and Minicozzi~\cite[Thm.~0.5]{cm2} and Espinar~\cite[Thm.~4.1]{espinar} for self-shrinker hypersurfaces. The corollary also holds for Gaussian perturbations $h=-r^2/2+g$ where $g$ is a concave function with bounded gradient.
\end{example}

\begin{remark}
An embedded two-sided hypersurface $P\sub M^m_w$ is \emph{weighted area-minimizing} if any relatively compact domain $P'\subeq P$ minimizes the area functional $A_h$ in \eqref{eq:volarea} among hypersurfaces with the same boundary. In particular, $P$ is strongly $h$-stable with $H^h_P=0$. So, the results in this section entail characterization and non-existence of properly embedded weighted area-minimizing hypersurfaces. In Gauss space this was done by Ca\~nete, Miranda and Vittone~\cite[Re.~2.7]{cvm}.
\end{remark}

\providecommand{\bysame}{\leavevmode\hbox to3em{\hrulefill}\thinspace}
\providecommand{\MR}{\relax\ifhmode\unskip\space\fi MR }
\providecommand{\MRhref}[2]{
  \href{http://www.ams.org/mathscinet-getitem?mr=#1}{#2}
}
\providecommand{\href}[2]{#2}

\end{document}